\documentclass[11pt]{article}

\usepackage[margin=1in]{geometry}

\usepackage{amsmath,amssymb,amsthm,mathtools}
\numberwithin{equation}{section}
\usepackage{mathrsfs}
\usepackage{enumitem}
\usepackage{tikz-cd}

\usepackage[hyperfootnotes=false]{hyperref}
\usepackage[nameinlink]{cleveref}

\hypersetup{
  colorlinks=true,
  citecolor=blue,
  linkcolor=blue,
  urlcolor=blue,
  pdftitle={Tropicalization and Polyhedral Completions of Weyl Algebras},
  pdfauthor={Jen-Chieh Hsiao},
  pdfkeywords={Weyl algebra, tropical D-ideal, nonarchimedean seminorm, polyhedral completion, flatness, ringed G-site}
}

\theoremstyle{plain}
\newtheorem{theorem}{Theorem}[section]
\newtheorem{proposition}[theorem]{Proposition}
\newtheorem{lemma}[theorem]{Lemma}
\newtheorem{corollary}[theorem]{Corollary}

\theoremstyle{definition}

\theoremstyle{remark}
\newtheorem{remark}[theorem]{Remark}

\title{Tropicalization and Polyhedral Completions of Weyl Algebras}

\author{Jen-Chieh Hsiao}
\date{}

\begin{document}

\maketitle
\begingroup
\renewcommand{\thefootnote}{}
\makeatletter
\renewcommand{\@makefntext}[1]{\noindent #1}
\makeatother
\footnotetext{\raggedright
Department of Mathematics, National Cheng Kung University, Tainan 701, Taiwan.\\
Email: \href{mailto:jhsiao@mail.ncku.edu.tw}{\nolinkurl{jhsiao@mail.ncku.edu.tw}}.\\
\href{https://orcid.org/0000-0001-7244-8312}{\color{black}ORCID: 0000-0001-7244-8312}.}
\endgroup
\setcounter{footnote}{0}

\begin{abstract}
In nonarchimedean toric geometry, compact rational polyhedra give rise to
affinoid domains and integral models. Motivated by this picture, we construct
weighted and polyhedral completions of the Weyl algebra \(D=D_n(K)\) over a
complete nonarchimedean field. Compact polyhedra \(\Delta\) in the admissible
weight region \(\Omega\subseteq\mathbb R^{2n}\) yield Banach completions
\(D_\Delta\), while arbitrary polyhedra yield
Fr\'echet completions by compact exhaustion. We prove that these algebras glue
over polyhedral coverings to form a structure sheaf on a polyhedral \(G\)-site
in the seminorm space of \(D\). We also establish that the tropical sets of
left ideals are supports of finite \(\Gamma\)-rational polyhedral complexes,
where \(\Gamma\) is the value group of \(K\), and determine natural subsites.

When \(K\) is discretely valued, we prove that the completions attached to
\(\Gamma\)-rational polyhedra are flat over \(D\) on both sides. This
establishes an exact base-change functor from finitely generated left
\(D\)-modules to sheaves on the \(\Gamma\)-rational subsite. For compact
integral \(\Delta\) in the interior of \(\Omega\), we further construct a
finite-type integral model whose completion is the unit ball of \(D_\Delta\),
prove that its special fiber is a toric face ring, and describe the reduction
of Gauss points in terms of its toric strata.
\end{abstract}

\begingroup

\endgroup

\section{Introduction}\label{sec:introduction}

Nonarchimedean analytic geometry describes analytic points by multiplicative
seminorms, while tropicalization records part of this information in
polyhedral form \cite{Berkovich90,GublerGuide}. In the toric setting, compact
rational polyhedra determine affinoid domains with explicit algebras of
analytic functions and integral models whose reductions reflect their face
structure \cite{Gubler07,GublerGuide}. Foster and Payne further organize
tropical polyhedral domains into a \(G\)-topology with structure sheaves
\cite{FosterPayne}. Section~\ref{subsec:toric-background} recalls the parts of
this commutative picture used below.

In this article, we develop an analogous polyhedral analytic framework for the
Weyl algebra. Weighted filtrations are standard in \(D\)-module theory, where
they control associated graded modules and characteristic varieties
\cite{SST00,ACJG00}. Over a nonarchimedean field, suitable weights also define
multiplicative Gauss norms and corresponding completions of the Weyl algebra
\cite{Pangalos}. We show that, as the weight varies, these single-weight
completions extend naturally to completions attached to polyhedra in the
weight space, fit together by restriction and gluing, and support exact base
change for finitely generated \(D\)-modules under the hypotheses stated below.

More precisely, let \(K\) be a complete nonarchimedean field, let
\(D=D_n(K)\) be the \(n\)-th Weyl algebra, and write
\(\operatorname{val}=-\log|\cdot|\) and \(\Gamma=\operatorname{val}(K^*)\).
The weighted Gauss norms are multiplicative on the admissible region
\[
\Omega=\{(u,v)\in\mathbb R^{2n}:u_i+v_i\leq0\text{ for all }i\}.
\]
For \(w\in\Omega\), we write \(D_w\) for the corresponding Banach
completion. A nonempty compact polyhedron \(\Delta\subseteq\Omega\)
determines the Banach completion \(D_\Delta\) for the supremum of these
norms, while an arbitrary nonempty polyhedron in \(\Omega\) determines a
Fr\'echet algebra \(D_\Delta^{\mathrm{Fr}}\) by compact exhaustion.

We describe the bounded spectra of the Banach completions in the space
\(\mathcal S(D)\) of nonarchimedean multiplicative seminorms on \(D\)
extending the absolute value on \(K\). Evaluation on the standard generators
defines a tropicalization map \(\operatorname{trop}:\mathcal S(D)\to\Omega\),
and the weighted Gauss norms define a continuous section. For a nonempty
compact polyhedron \(\Delta\subseteq\Omega\),
Theorem~\ref{thm:compact-spectrum} identifies the bounded spectrum of
\(D_\Delta\) with
\[
U_{\Delta^\uparrow}=\operatorname{trop}^{-1}(\Delta^\uparrow),
\qquad
\Delta^\uparrow=(\Delta+\mathbb R_{\geq0}^{2n})\cap\Omega.
\]
We also show that the Gauss points over the vertices of \(\Delta\) compute
the defining norm of \(D_\Delta\). Thus the compact polyhedral completion has
an explicit bounded spectrum and a finite norming set.

The main algebraic result is two-sided flatness. When \(K\) is discretely
valued, we prove that \(D\to D_\Delta\) is flat for every nonempty compact
\(\Gamma\)-rational polyhedron \(\Delta\subseteq\Omega\)
(Theorem~\ref{thm:compact-flat}), and that
\(D\to D_\Delta^{\mathrm{Fr}}\) is flat for every nonempty
\(\Gamma\)-rational polyhedron in \(\Omega\)
(Theorem~\ref{thm:frechet-flat}). For compact polyhedra, the proof passes
through finite-type models over the valuation ring and their adic completions;
the general case is obtained from compact exhaustion by an inverse-limit
argument. We also prove that, for a finite presentation \(M\simeq D^r/N\),
the submodule \(D_\Delta^{\mathrm{Fr}}N\) is closed in
\((D_\Delta^{\mathrm{Fr}})^r\). Together these results show that base change
to \(D_\Delta^{\mathrm{Fr}}\) is exact and admits a compatible
projective-limit description over compact exhaustions
(Corollary~\ref{cor:exact-base-change}).

The polyhedral completions also fit naturally into a sheaf-theoretic
structure. The assignments
\[
U_\Delta\longmapsto D_\Delta^{\mathrm{Fr}},
\qquad U_\Delta=\operatorname{trop}^{-1}(\Delta),
\]
form a structure sheaf on the polyhedral \(G\)-site
(Theorem~\ref{thm:structure-sheaf}). When \(K\) is discretely valued,
flatness then allows finitely generated left \(D\)-modules to extend over the
\(\Gamma\)-rational subsite: every such module \(M\) determines a sheaf with
sections \(D_\Delta^{\mathrm{Fr}}\otimes_D M\), and the resulting functor is
exact (Theorem~\ref{thm:module-sheaf-exact}). Here the module-sheaf gluing is
not automatic from the sheaf property of the section algebras. For a binary
compact covering, we establish a Mayer--Vietoris sequence for the polyhedral
completions and use flatness to preserve its exactness after tensoring with
\(M\). The general finite-covering case then follows by refining the cover to
a polyhedral subdivision and applying the binary gluing argument repeatedly.
Arbitrary admissible coverings are handled by passing to compatible compact
subdomains.

The finite-type models used for compact flatness also carry a geometric
reduction theory. Assume that \(K\) is discretely valued, with valuation ring
\(R\), uniformizer \(\pi\), and residue field \(k=R/\pi R\), and normalize
the valuation by \(\operatorname{val}(\pi)=1\). Let \(\Delta\) be a nonempty
compact polyhedron with integral vertices in the interior \(\Omega^\circ\) of
\(\Omega\). We construct a finite-type algebra \(A_\Delta\) over \(R\) whose
\(\pi\)-adic completion is the unit ball \(D_\Delta^\circ\), and prove that
its special-fiber algebra is the toric face ring
\[
A_\Delta/\pi A_\Delta\simeq k[\Sigma_{\Delta^+}],
\]
where \(\Sigma_{\Delta^+}\) is the minimum normal fan of the upper hull
\(\Delta^+=\Delta+\mathbb R_{\geq0}^{2n}\)
(Theorem~\ref{thm:toric-face-reduction}). We then define a reduction map from
the bounded spectrum of \(D_\Delta\) to \(\operatorname{Spec}k[\Sigma_{\Delta^+}]\)
by sending a seminorm \(\eta\) to the prime ideal
\[
\{\overline f:f\in D_\Delta^\circ,\ \eta(f)<1\},
\]
where \(\overline f\) denotes the image of \(f\) in
\(D_\Delta^\circ/\pi D_\Delta^\circ\simeq k[\Sigma_{\Delta^+}]\). For the
Gauss point associated with \(w\in\Delta^\uparrow\), we show that this prime
is the torus-invariant prime determined by the normal cone of the face of
\(\Delta^+\) whose relative interior contains \(w\)
(Theorem~\ref{thm:reduction-map} and Corollary~\ref{cor:gauss-toric-strata}).
Thus the position of \(w\) in \(\Delta^+\) determines both its analytic Gauss
point and its algebraic reduction.

We return to the tropicalization of \(D\)-ideals studied in
\cite{HsiaoTropD}. For a left ideal \(I\subseteq D\), we define
\(\operatorname{Trop}_D(I)\) as the closure in \(\Omega\) of the
monomial-free tropical locus in \(\Omega^\circ\), and prove that it is the
support of a finite \(\Gamma\)-rational polyhedral complex
(Theorem~\ref{thm:trop-ideal-polyhedral}). The \(\Gamma\)-rational polyhedra
contained in this tropical set therefore determine a natural subsite of the
\(\Gamma\)-rational polyhedral site. A left ideal need not be two-sided, so
\(D/I\) is naturally a left \(D\)-module but need not be a quotient algebra.
Under the discrete-valuation hypothesis, the ideal-dependent analytic object
is the restriction of the cyclic module sheaf \(\widetilde{D/I}\) to this
subsite.

The integral models and flatness results developed in this article also
underlie our ongoing work on characteristic varieties and cycles of
\(D\)-modules over discretely valued fields. Using integral reduction, we
construct weighted characteristic cycles and show that the tropicalizations
of their restrictions to the coordinate torus assemble into a balanced
tropical cycle on \(\Omega^\circ\). In one variable and equicharacteristic
zero, we further show that this invariant, together with the rank, determines
the intrinsic subsidiary-radius multiset of a differential module over
\(K(T)\) as the Gauss norm varies.

We mention here some related strands of work. Soibelman develops
noncommutative analytic spaces over nonarchimedean fields, including
noncommutative analytic affine spaces and quantum analytic tori
\cite{Soibelman09}. In a different direction, Ardakov and Wadsley construct
sheaves of completed differential operators on smooth rigid analytic spaces,
whose local section algebras are Fr\'echet--Stein, together with the
corresponding theory of coadmissible modules \cite{ArdakovWadsley19}. These
works provide broader background for the nonarchimedean Banach and Fr\'echet
structures considered here.

The paper is organized as follows. Section~\ref{sec:preliminaries} recalls
the nonarchimedean, polyhedral, and sheaf-theoretic background together with
the toric constructions used for comparison. Sections~\ref{sec:seminorms}
and~\ref{sec:polyhedral-completions} develop the seminorm space and the
weighted and polyhedral completions. Section~\ref{sec:flatness} proves the
flatness and base-change results, and Section~\ref{sec:integral-models}
develops the integral models and toric-face reduction.
Section~\ref{sec:polyhedral-site} constructs the polyhedral \(G\)-site and its
structure sheaves. Section~\ref{sec:module-sheaves} constructs the associated
module sheaves and returns to tropical \(D\)-ideals.

\section{Preliminaries}\label{sec:preliminaries}

Throughout, \(K\) denotes a complete nonarchimedean field with absolute
value \(|\cdot|\). We write
\[
\operatorname{val}(a)=-\log|a|,
\qquad
\Gamma=\operatorname{val}(K^*)\subseteq\mathbb R,
\]
and allow the absolute value to be trivial. We set
\(\operatorname{val}(0)=+\infty\).

\subsection{Nonarchimedean Banach and Fr\'echet algebras}
\label{subsec:banach-frechet-background}

A nonarchimedean norm on a \(K\)-vector space satisfies
\[
\|av\|=|a|\|v\|,
\qquad
\|v+v'\|\leq\max\{\|v\|,\|v'\|\}.
\]
A Banach \(K\)-algebra is a complete normed \(K\)-algebra whose norm is
submultiplicative.

For a Banach \(K\)-algebra \(A\), not necessarily commutative, let
\(\mathcal M(A)\) denote the space of bounded nonarchimedean multiplicative
seminorms on \(A\) extending \(|\cdot|\), endowed with the topology of
pointwise convergence. Here bounded means that \(\eta(a)\leq C\|a\|\)
for some \(C>0\) and all \(a\in A\).

Every \(\eta\in\mathcal M(A)\) is contractive. Indeed, applying this
inequality to \(a^m\), we obtain
\[
\eta(a)^m\leq C\|a^m\|\leq C\|a\|^m.
\]
Taking \(m\)-th roots and letting \(m\to\infty\) yields
\(\eta(a)\leq\|a\|\). Consequently, \(\mathcal M(A)\) is a closed
subset of \(\prod_{a\in A}[0,\|a\|]\) and is therefore compact.

The Fr\'echet \(K\)-algebras occurring in this paper are countable projective
limits \(A=\varprojlim_j A_j\) of Banach \(K\)-algebras with continuous
transition homomorphisms, endowed with the topology defined by the induced
Banach seminorms. We use no general theory of Fr\'echet algebras beyond this
concrete setting. For background on nonarchimedean locally convex spaces,
see \cite[Chapter~I]{SchneiderNFA}.

\subsection{Polyhedra and normal fans}\label{subsec:polyhedral-background}

A polyhedron in \(\mathbb R^m\) is an intersection of finitely many closed
half-spaces. We allow the empty polyhedron when discussing polyhedral sites.
A polyhedron is \(\Gamma\)-rational if it can be written as
\[
\Delta=\{w\in\mathbb R^m:
\langle a_j,w\rangle\leq b_j,\ j=1,\ldots,r\},
\]
with \(a_j\in\mathbb Z^m\) and \(b_j\in\Gamma\). We regard the empty
polyhedron as \(\Gamma\)-rational. A compact polyhedron is called a
polytope, and \(\operatorname{Vert}(\Delta)\) denotes its set of vertices.

A nonempty face of a polyhedron is the locus on which a linear functional
attains its minimum; we also allow the empty face. We write
\(\operatorname{relint}(F)\) for the interior of \(F\) in its affine span.
A finite polyhedral complex \(\Sigma\) is a finite collection of polyhedra
such that every face of a member belongs to \(\Sigma\), and the intersection
of two members is either empty or a common face. Its support is
\(|\Sigma|=\bigcup_{\sigma\in\Sigma}\sigma\). A subcomplex is a
subcollection that is itself a polyhedral complex. The complex is
\(\Gamma\)-rational if all of its cells are \(\Gamma\)-rational. Its
support need not be convex.

For a nonempty polyhedron \(\Delta\subseteq\mathbb R^m\), its recession
cone is
\[
\operatorname{rec}(\Delta)=\{r\in\mathbb R^m:
 w+tr\in\Delta\text{ for every }w\in\Delta,\ t\geq0\}.
\]
For a cone \(\sigma\subseteq\mathbb R^m\), write
\(\sigma^\vee=\{a\in(\mathbb R^m)^*:
\langle a,r\rangle\geq0\text{ for every }r\in\sigma\}\).
We use the minimum convention for normal fans. For a nonempty face \(F\)
of \(\Delta\), its normal cone is
\[
N_\Delta(F)=\{a\in(\mathbb R^m)^*:
\langle a,w\rangle\leq\langle a,w'\rangle
\text{ for every }w\in F,\ w'\in\Delta\}.
\]
These cones form the normal fan \(\Sigma_\Delta\), whose support is
\[
|\Sigma_\Delta|=\operatorname{rec}(\Delta)^\vee.
\]
In particular, the normal fan need not be complete when \(\Delta\) is
unbounded. With this convention,
\(a\in\operatorname{relint}N_\Delta(F)\) precisely when \(F\) is the
minimizing face of \(a\).

We will also use compact polyhedral exhaustions. A nested sequence
\(C_1\subseteq C_2\subseteq\cdots\subseteq\Delta\) of nonempty compact
polyhedra is cofinal in \(\Delta\) if every compact polyhedron contained in
\(\Delta\) is contained in some \(C_j\). Every nonempty polyhedron admits
such an exhaustion; for example, after omitting empty initial terms, one
may take \(C_j=\Delta\cap[-j,j]^m\). When \(\Gamma=\mathbb Z\) and
\(\Delta\) is \(\Gamma\)-rational, we will always choose the exhaustion
to be \(\Gamma\)-rational, as in this example. Finite intersections of
\(\Gamma\)-rational polyhedra are again \(\Gamma\)-rational.

For background on polyhedra, polyhedral complexes, normal fans, and
recession cones, see \cite[Sections~2.3 and~3.5]{MS15}.

\subsection{\texorpdfstring{\(G\)}{G}-sites, sheaves, and gluing}
\label{subsec:sites-sheaves-background}

Let \(\mathcal C\) be a category admitting the fiber products used below.
A Grothendieck pretopology specifies, for each object \(U\), certain
families \(\{U_i\to U\}_{i\in I}\), called covering families, subject
to the following conditions:
\begin{enumerate}[label=\textup{(\roman*)}]
\item Every isomorphism \(V\to U\) is a covering.
\item If \(\{U_i\to U\}\) is a covering and \(V\to U\) is a morphism,
then \(\{V\times_U U_i\to V\}_{i\in I}\) is a covering.
\item If \(\{U_i\to U\}\) is a covering and each \(U_i\) has a covering
\(\{V_{ij}\to U_i\}_{j\in J_i}\), then the composite family
\(\{V_{ij}\to U\}_{i,j}\) is a covering.
\end{enumerate}
We use the associated Grothendieck topology and refer to \(\mathcal C\)
as a \(G\)-site.

A presheaf \(\mathscr F\) on \(\mathcal C\) is a contravariant functor.
For a covering \(\{U_i\to U\}\), write \(U_{ij}=U_i\times_U U_j\).
A family \(s_i\in\mathscr F(U_i)\) is compatible if the restrictions of
\(s_i\) and \(s_j\) to \(U_{ij}\) agree for every \(i,j\). The
presheaf is a sheaf if every compatible family is the restriction of a
unique section in \(\mathscr F(U)\).

For a presheaf of abelian groups, this condition is equivalently the
exactness of
\begin{equation}\label{eq:finite-gluing-property}
0\longrightarrow\mathscr F(U)
\longrightarrow\prod_i\mathscr F(U_i)
\xrightarrow{\delta}\prod_{i,j}\mathscr F(U_{ij}),
\end{equation}
where \(\delta((s_i)_i)=(s_i|_{U_{ij}}-s_j|_{U_{ij}})_{i,j}\).
For a specified finite covering, we say that \(\mathscr F\) has the
\emph{gluing property} if this sequence is exact. We use the same
terminology for presheaves of rings or modules by passing to their
underlying additive groups.

A ringed \(G\)-site is a pair \((\mathcal C,\mathcal O)\), where
\(\mathcal O\) is a sheaf of rings. No commutativity is assumed. A sheaf
of left \(\mathcal O\)-modules is a sheaf \(\mathscr M\) such that each
\(\mathscr M(U)\) is a left \(\mathcal O(U)\)-module and restrictions
satisfy \((fm)|_V=f|_V\,m|_V\) for every morphism \(V\to U\).

When \(\mathcal D\subseteq\mathcal C\) is a full subcategory stable under
the fiber products occurring in its inherited coverings, we equip it with
the induced pretopology. The restriction of a sheaf on \(\mathcal C\) is
then a sheaf on \(\mathcal D\). In particular, a ringed \(G\)-site
restricts to a ringed subsite in this way.

For the sheaves of Banach or Fr\'echet \(K\)-algebras considered below,
the restriction homomorphisms are continuous. The topology is additional
objectwise structure; the sheaf condition is the algebraic condition above.
For general background on sites and sheaves, see
\cite[Chapters~16--17 and Section~18.1]{KashiwaraSchapira}. For \(G\)-topologies in
nonarchimedean geometry, see \cite[Sections~9.1--9.2]{BGR} and
\cite{Berkovich90}.

\subsection{Toric polytopal domains and tropical \texorpdfstring{\(G\)}{G}-sites}
\label{subsec:toric-background}

Let \(T=\operatorname{Spec}K[M]\) be a split torus, put
\(N_{\mathbb R}=\operatorname{Hom}(M,\mathbb R)\), and write \(\chi^m\)
for the character associated with \(m\in M\). Tropicalization is the
continuous map
\[
\operatorname{trop}:T^{\mathrm{an}}\longrightarrow N_{\mathbb R},
\qquad\langle m,\operatorname{trop}(p)\rangle=-\log|\chi^m(p)|.
\]
For \(w\in N_{\mathbb R}\), the Gauss formula
\[
s(w)\!\left(\sum_m a_m\chi^m\right)
=\max_m|a_m|e^{-\langle w,m\rangle}
\]
defines a multiplicative seminorm and a section of tropicalization.

Let \(\Delta\subseteq N_{\mathbb R}\) be a nonempty compact
\(\Gamma\)-rational polyhedron. Then
\(U_\Delta^{\mathrm{com}}=\operatorname{trop}^{-1}(\Delta)\) is a
polytopal affinoid domain \cite[Proposition~4.1]{Gubler07}. We use the
superscript \({\rm com}\) to distinguish the toric domains and section
algebras in this subsection from the analogous ones for the Weyl algebra
introduced later. The algebra of analytic functions on
\(U_\Delta^{\mathrm{com}}\) is
\[
A_\Delta^{\mathrm{com}}=
\left\{\sum_{m\in M}a_m\chi^m:
\operatorname{val}(a_m)+\min_{w\in\Delta}\langle w,m\rangle
\longrightarrow+\infty\text{ as }\|m\|\to\infty\right\},
\]
where \(\|\cdot\|\) is any norm on
\(M_{\mathbb R}=M\otimes_{\mathbb Z}\mathbb R\). Its supremum norm
satisfies
\[
\|f\|_\Delta=\sup_{p\in U_\Delta^{\mathrm{com}}}|f(p)|
=\sup_{w\in\Delta}s(w)(f)
=\max_{q\in\operatorname{Vert}(\Delta)}s(q)(f).
\]
Thus \(s(\Delta)\), and already the Gauss points over the vertices of
\(\Delta\), is \emph{norming} for \(A_\Delta^{\mathrm{com}}\): it computes
the supremum norm of every analytic function on \(U_\Delta^{\mathrm{com}}\).
See
\cite[Propositions~6.17 and~6.19, Lemma~6.21]{GublerGuide} for these
descriptions and identifications.

Write \(K^\circ=\{a\in K:|a|\leq1\}\) and
\(K^{\circ\circ}=\{a\in K:|a|<1\}\) for the valuation ring and its
maximal ideal, and let \(\widetilde K=K^\circ/K^{\circ\circ}\) be the
residue field. The unit ball \((A_\Delta^{\mathrm{com}})^\circ\) for
\(\|\cdot\|_\Delta\) consists of the series above satisfying
\(\operatorname{val}(a_m)+\min_{w\in\Delta}\langle w,m\rangle\geq0\)
for all \(m\). Its Laurent polynomials form Gubler's integral algebra
\[
K[M]^\Delta=\left\{\sum_m a_m\chi^m\in K[M]:
\operatorname{val}(a_m)+\min_{w\in\Delta}\langle w,m\rangle\geq0
\text{ for all }m\right\}.
\]
The associated polyhedral scheme is
\(\mathcal U_\Delta=\operatorname{Spec}K[M]^\Delta\)
\cite[Section~6.11]{GublerGuide}, with special fiber
\[
\widetilde{\mathcal U}_\Delta=
\operatorname{Spec}\!\left(K[M]^\Delta/K^{\circ\circ}K[M]^\Delta\right).
\]
The reduction map is
\[
\operatorname{red}_\Delta:U_\Delta^{\mathrm{com}}
\longrightarrow\widetilde{\mathcal U}_\Delta,
\qquad
p\longmapsto\{\overline f:f\in K[M]^\Delta,\ |f(p)|<1\},
\]
where \(\overline f\) denotes the class of \(f\) modulo
\(K^{\circ\circ}K[M]^\Delta\); see
\cite[Sections~4.8--4.9 and Lemma~6.21]{GublerGuide}.
By \cite[Proposition~6.22]{GublerGuide}, the torus orbits \(Z\) of
\(\widetilde{\mathcal U}_\Delta\) are in order-reversing bijection with
the open faces \(\tau=\operatorname{relint}(F)\) of \(\Delta\), both
ordered by inclusion of closures. This correspondence is characterized by
\[
Z=\operatorname{red}_\Delta\!\left(\operatorname{trop}^{-1}(\tau)\right),
\qquad
\tau=\operatorname{trop}\!\left(\operatorname{red}_\Delta^{-1}(Z)\right).
\]
These are the aspects of the commutative construction used for comparison
with the integral models and reduction map of
Section~\ref{sec:integral-models}.

For an unbounded \(\Gamma\)-rational polyhedron containing no affine line,
the toric construction uses its recession cone to determine an affine toric
chart. The associated affinoid domain lies in that chart, and its
intersection with \(T^{\mathrm{an}}\) is the tropical inverse image of the
original polyhedron; see \cite[Propositions~6.17 and~6.19, Lemma~6.21]{GublerGuide}.
Our construction for the Weyl algebra keeps the ambient seminorm space fixed and treats
arbitrary polyhedra by compact exhaustion.

Over a complete algebraically closed nontrivially valued field, Foster and
Payne equip tropicalizations with a \(G\)-topology whose basic polyhedral
domains pull back to analytic domains and with an integral structure sheaf
arising from Gubler models; see \cite[Section~5.1]{FosterPayne}.
Sections~\ref{sec:polyhedral-site} and~\ref{sec:module-sheaves} construct
the polyhedral \(G\)-site and its algebra and module sheaves for the Weyl
algebra.

\section{Seminorms and tropicalization for the Weyl algebra}
\label{sec:seminorms}

We begin with multiplicative seminorms and the weighted Gauss norms determined
by the values of the standard generators. The
resulting admissible weight region will parametrize all completions considered
in the following sections.

\subsection{The seminorm space and admissible region}

Let \(D=D_n(K)\) be the \(n\)-th Weyl algebra. Every \(P\in D\) has a
unique finite expansion
\[
P=\sum_{\alpha,\beta\in\mathbb N^n}
 c_{\alpha,\beta}x^\alpha\partial^\beta.
\]
We call the elements \(x^\alpha\partial^\beta\) the \emph{ordered
monomials}. The weighted norms and completions below are defined with respect
to this expansion.

Let \(\mathcal S(D)\) be the set of nonarchimedean multiplicative seminorms
\(\eta:D\to\mathbb R_{\geq0}\) extending the absolute value on \(K\).
Thus
\[
\eta(PQ)=\eta(P)\eta(Q),\qquad
\eta(P+Q)\leq\max\{\eta(P),\eta(Q)\},\qquad
\eta(a)=|a|
\]
for \(P,Q\in D\) and \(a\in K\). We equip \(\mathcal S(D)\) with the
topology of pointwise convergence. Since \(D\) is noncommutative, the
nonarchimedean inequality is included explicitly in the definition.

For \(\eta\in\mathcal S(D)\), the relation
\(\partial_i x_i=x_i\partial_i+1\) implies
\(\eta(x_i)\eta(\partial_i)\geq1\). Indeed, if this product were less
than \(1\), then \(\eta(x_i\partial_i)<1=\eta(1)\), so
\(\eta(x_i\partial_i+1)=1\), whereas multiplicativity would imply
\(\eta(\partial_i x_i)<1\). We therefore set
\[
\Omega=\{(u,v)\in\mathbb R^{2n}:u_i+v_i\leq0\text{ for all }i\}
\]
and define
\[
\operatorname{trop}:\mathcal S(D)\longrightarrow\Omega,
\qquad
\eta\longmapsto
\bigl(-\log\eta(x_1),\ldots,-\log\eta(x_n),
-\log\eta(\partial_1),\ldots,-\log\eta(\partial_n)\bigr).
\]
This map is continuous by the definition of the topology on \(\mathcal S(D)\).

\subsection{Weighted Gauss norms and the Gauss section}

For \(\varepsilon=(\varepsilon_1,\ldots,\varepsilon_{2n})\in
\mathbb R_{>0}^{2n}\), define
\[
\left|\sum_{\alpha,\beta}c_{\alpha,\beta}x^\alpha\partial^\beta
\right|_\varepsilon
=\max_{\alpha,\beta}|c_{\alpha,\beta}|
\varepsilon_1^{\alpha_1}\cdots\varepsilon_n^{\alpha_n}
\varepsilon_{n+1}^{\beta_1}\cdots\varepsilon_{2n}^{\beta_n}.
\]
This is a nonarchimedean norm on the underlying \(K\)-vector space of \(D\).
The following multiplicativity criterion is due to Pangalos
\cite[Lemma~1.2.4]{Pangalos}.

\begin{proposition}\label{prop:weighted-gauss-multiplicative}
The norm \(|\cdot|_\varepsilon\) is multiplicative if and only if
\[
\varepsilon_i\varepsilon_{n+i}\geq1\qquad(1\leq i\leq n).
\]
\end{proposition}

For \(w=(u,v)\in\Omega\), take \(\varepsilon_i=e^{-u_i}\) and
\(\varepsilon_{n+i}=e^{-v_i}\). Since \(u_i+v_i\leq0\),
Proposition~\ref{prop:weighted-gauss-multiplicative} implies that
\[
\left|\sum_{\alpha,\beta}c_{\alpha,\beta}x^\alpha\partial^\beta\right|_w
=\max_{\alpha,\beta}|c_{\alpha,\beta}|
e^{-\langle u,\alpha\rangle-\langle v,\beta\rangle}
\]
is multiplicative. These norms provide a section of tropicalization.

\begin{proposition}\label{prop:gauss-section}
The map
\[
s:\Omega\longrightarrow\mathcal S(D),\qquad w\longmapsto|\cdot|_w,
\]
is a continuous section of \(\operatorname{trop}\). Moreover, if
\(w=\operatorname{trop}(\eta)\), then
\[
\eta(P)\leq s(w)(P)=|P|_w\qquad(P\in D).
\]
Thus \(s(w)\) is the greatest seminorm in the fiber
\(\operatorname{trop}^{-1}(w)\) for the pointwise order.
\end{proposition}

\begin{proof}
The equalities \(|x_i|_w=e^{-u_i}\) and \(|\partial_i|_w=e^{-v_i}\)
imply \(\operatorname{trop}\circ s=\operatorname{id}_\Omega\). For fixed
\(P\), the function \(w\mapsto|P|_w\) is the maximum of finitely many
exponential affine functions, so \(s\) is continuous. If
\(w=\operatorname{trop}(\eta)\) and
\(P=\sum c_{\alpha,\beta}x^\alpha\partial^\beta\), then
\[
\eta(P)\leq\max_{\alpha,\beta}|c_{\alpha,\beta}|
\eta(x)^\alpha\eta(\partial)^\beta=|P|_w.
\]
This proves the last assertion.
\end{proof}

For every \(\Delta\subseteq\Omega\), the restriction \(s|_\Delta\) is a
homeomorphism onto its image, with inverse induced by tropicalization. In
particular, \(s(\Delta)=s(\Omega)\cap\operatorname{trop}^{-1}(\Delta)\).
We will use these Gauss points later to compute the norms of compact
polyhedral completions.

\subsection{Weighted filtrations and tropical \texorpdfstring{\(D\)}{D}-ideals}
\label{subsec:weighted-initials}

For \(w=(u,v)\in\Omega\) and nonzero
\(P=\sum c_{\alpha,\beta}x^\alpha\partial^\beta\in D\), set
\[
\nu_w(P)=\min_{\alpha,\beta}
\bigl(\operatorname{val}(c_{\alpha,\beta})
+\langle u,\alpha\rangle+\langle v,\beta\rangle\bigr),
\]
and put \(\nu_w(0)=+\infty\). By
Proposition~\ref{prop:weighted-gauss-multiplicative},
\(\nu_w(PQ)=\nu_w(P)+\nu_w(Q)\). It therefore defines the filtration
\[
F_{\geq\gamma}D=\{P\in D:\nu_w(P)\geq\gamma\},\qquad\gamma\in\mathbb R.
\]
We write \(\operatorname{gr}_w(D)\) for the associated graded ring and
\(\sigma_w(P)\) for the principal symbol of a nonzero element \(P\).
For a left ideal \(I\subseteq D\), define
\[
\operatorname{in}_w(I)
=\bigl\langle\sigma_w(P):0\neq P\in I\bigr\rangle
\subseteq\operatorname{gr}_w(D).
\]
With the induced filtration on \(I\), this is the associated graded ideal
\(\operatorname{gr}_w(I)\).

On the strict interior
\[
\Omega^\circ=\{(u,v)\in\mathbb R^{2n}:u_i+v_i<0\text{ for all }i\},
\]
all nonprincipal terms arising from the commutation relations have strictly
larger filtered value. Hence
\[
\operatorname{gr}_w(D)\simeq
\operatorname{gr}(K)[x_1,\ldots,x_n,\xi_1,\ldots,\xi_n],
\qquad w\in\Omega^\circ,
\]
where \(\operatorname{gr}(K)\) is the graded field associated with the
valuation on \(K\).

For a left ideal \(I\subseteq D\), define its \emph{interior tropicalization} by
\begin{equation}\label{eq:trop-D-ideal-interior}
\operatorname{Trop}_D^\circ(I)
=\{w\in\Omega^\circ:\operatorname{in}_w(I)\text{ contains no monomial}\}.
\end{equation}
Here a monomial means an element \(a x^\alpha\xi^\beta\), where \(a\) is a
nonzero homogeneous element of \(\operatorname{gr}(K)\). We define
\begin{equation}\label{eq:trop-D-ideal}
\operatorname{Trop}_D(I)
=\overline{\operatorname{Trop}_D^\circ(I)}^{\,\Omega},
\end{equation}
where the closure is taken in \(\Omega\). Thus boundary points are included by
closure from the strict interior.

The finite tropical-basis theorem on the strict region gives the
polyhedrality needed below.

\begin{theorem}\label{thm:trop-ideal-polyhedral}
For every left ideal \(I\subseteq D\), the set
\(\operatorname{Trop}_D(I)\) is the support of a finite
\(\Gamma\)-rational polyhedral complex.
\end{theorem}

\begin{proof}
By the argument of \cite[Theorem~5.2]{HsiaoTropD}, with the sign convention
for weights reversed, there are finitely many \(P_1,\ldots,P_s\in I\) such
that
\begin{equation}\label{eq:tropical-basis}
\operatorname{Trop}_D^\circ(I)
=\bigcap_{j=1}^s\operatorname{Trop}_D^\circ(P_j).
\end{equation}
We use this finite tropical-basis statement over the present ground field.
Although \cite{HsiaoTropD} is formulated in characteristic zero, the ordered-
monomial multiplication formula has principal coefficient~\(1\) and integral
correction coefficients, while the Hilbert-function, monomial-ideal, and
weighted linear-algebra arguments used in the proof are characteristic-free;
see \cite[Proposition~3.1 and Lemmas~3.8, 4.2, and~4.3]{HsiaoTropD}. The
proof of \cite[Theorem~5.2]{HsiaoTropD} already removes the auxiliary splitting
hypothesis by passing to a valued field extension admitting a splitting and
descending the resulting tropical basis via \cite[Lemma~5.1]{HsiaoTropD}.
For a nontrivial valuation, the normalization \(1\in\Gamma\) used there is
immaterial after rescaling the valuation and the weight space by a positive
element of \(\Gamma\); for the trivial valuation, the same argument reduces to
the usual unvalued Gr\"obner setting. Hence \eqref{eq:tropical-basis}
remains valid under the present hypotheses.

Write
\[
P_j=\sum_m c_m x^\alpha\partial^\beta,
\qquad m=(\alpha,\beta).
\]
Membership in \(\operatorname{Trop}_D^\circ(P_j)\) means that the minimum of
the finitely many affine functions
\[
w\longmapsto \operatorname{val}(c_m)+\langle w,m\rangle
\]
is attained at least twice. The equalities between pairs of these affine
functions form a finite \(\Gamma\)-rational hyperplane arrangement. Taking a
common refinement for \(P_1,\ldots,P_s\) and with the faces of \(\Omega\), we
obtain a finite \(\Gamma\)-rational polyhedral complex \(\Sigma\) in
\(\Omega\) such that \(\operatorname{Trop}_D^\circ(I)\) is a union of
relative interiors of cells of \(\Sigma\) contained in \(\Omega^\circ\).
The subcomplex generated by the closures of these cells therefore has support
\(\overline{\operatorname{Trop}_D^\circ(I)}^{\,\Omega}\), which is
\(\operatorname{Trop}_D(I)\) by definition.
\end{proof}

\section{Polyhedral completions and spectra}
\label{sec:polyhedral-completions}

We now pass from the weighted Gauss norms of Section~\ref{sec:seminorms} to
their completions. No rationality or integrality hypothesis is imposed in
this section.

\subsection{Weighted and compact completions}

For \(w\in\Omega\), let \(D_w\) be the completion of \(D\) with respect
to \(|\cdot|_w\). If \(\Delta\subseteq\Omega\) is a nonempty compact
polyhedron, set
\[
|P|_\Delta=\sup_{w\in\Delta}|P|_w
=\max_{q\in\operatorname{Vert}(\Delta)}|P|_q,
\]
where, for \(P\neq0\), the second equality follows because
\(w\mapsto\log|P|_w\) is convex and piecewise affine. Let \(D_\Delta\)
be the completion of \(D\) for this norm.

\begin{proposition}\label{prop:series-descriptions}
For \(w=(u,v)\in\Omega\),
\[
D_w=\left\{\sum_{\alpha,\beta}c_{\alpha,\beta}x^\alpha\partial^\beta:
|c_{\alpha,\beta}|e^{-\langle u,\alpha\rangle-\langle v,\beta\rangle}
\longrightarrow0\text{ as }|\alpha|+|\beta|\to\infty\right\}.
\]
The norm of such a series is the supremum of its weighted coefficients.
If \(\Delta\subseteq\Omega\) is nonempty and compact, then
\[
D_\Delta=\bigcap_{q\in\operatorname{Vert}(\Delta)}D_q
\]
inside the space of formal ordered-monomial series, and
\(|f|_\Delta=\max_{q\in\operatorname{Vert}(\Delta)}|f|_q\). In particular,
\(D_w\) and \(D_\Delta\) are complete nonarchimedean \(K\)-algebras.
\end{proposition}

\begin{proof}
Put \(\rho_{\alpha,\beta}=e^{-\langle u,\alpha\rangle-\langle v,\beta\rangle}\).
If \(P_j=\sum c_{\alpha,\beta}^{(j)}x^\alpha\partial^\beta\) is Cauchy
for \(|\cdot|_w\), then
\[
|c_{\alpha,\beta}^{(j)}-c_{\alpha,\beta}^{(k)}|\rho_{\alpha,\beta}
\leq|P_j-P_k|_w,
\]
so each coefficient sequence is Cauchy and hence converges in \(K\), since
\(K\) is complete. Since every \(P_j\) has finite support, the Cauchy
condition also implies that the limiting weighted coefficients
\(|c_{\alpha,\beta}|\rho_{\alpha,\beta}\) tend to zero. Conversely, any
formal series with this property is the \(|\cdot|_w\)-limit of its finite
truncations. This proves the description of \(D_w\) and its norm.

If \(q_1,\ldots,q_s\) are the vertices of \(\Delta\), then
\(|P|_\Delta=\max_i|P|_{q_i}\). Hence a sequence is Cauchy for
\(|\cdot|_\Delta\) if and only if it is Cauchy for each vertex norm.
The preceding description, using the same finite truncations at all
vertices, therefore implies
\[
D_\Delta=\bigcap_{i=1}^sD_{q_i}
\]
inside the space of formal ordered-monomial series, with the stated norm.
Multiplication extends continuously from \(D\).
\end{proof}

If \(\Delta'\subseteq\Delta\) are nonempty compact polyhedra, then
\(|P|_{\Delta'}\leq|P|_\Delta\), so the identity on \(D\) extends to a
continuous injective homomorphism \(D_\Delta\to D_{\Delta'}\). Under the
series descriptions, this is the natural inclusion. If \(f\in D_\Delta\)
and \(w\in\Delta\), we write \(|f|_w\) for the norm of its image in
\(D_w\).

\subsection{Bounded spectra}

For a polyhedron \(\Delta\subseteq\Omega\), write
\(U_\Delta=\operatorname{trop}^{-1}(\Delta)\).

\begin{proposition}\label{prop:weighted-spectra-cover}
For \(w\in\Omega\), restriction to \(D\) induces a homeomorphism
\[
\mathcal M(D_w)\xrightarrow{\;\sim\;}
\{\eta\in\mathcal S(D):\eta\leq|\cdot|_w\}
=\{\eta\in\mathcal S(D):\operatorname{trop}(\eta)\geq w\},
\]
where the last inequality is coordinatewise. In particular, the displayed
subset is compact, and
\[
\mathcal S(D)=\bigcup_{w\in\Omega}\mathcal M(D_w),
\]
where each \(\mathcal M(D_w)\) is viewed in \(\mathcal S(D)\) through this
homeomorphism.
\end{proposition}

\begin{proof}
By Section~\ref{subsec:banach-frechet-background}, every element of
\(\mathcal M(D_w)\) is contractive, so its restriction \(\eta\) to
\(D\) satisfies \(\eta\leq|\cdot|_w\).
Conversely, any \(\eta\in\mathcal S(D)\) with this property extends
uniquely to \(D_w\), since \(D\) is dense in \(D_w\); multiplicativity of
the extension follows by continuity. Thus restriction is a continuous
bijection onto the first displayed subset. Since \(\mathcal M(D_w)\) is
compact by Section~\ref{subsec:banach-frechet-background} and
\(\mathcal S(D)\) is Hausdorff, it is a homeomorphism onto its
image.

Write \(w=(u,v)\). The inequality \(\operatorname{trop}(\eta)\geq w\)
is equivalent to \(\eta(x_i)\leq e^{-u_i}\) and
\(\eta(\partial_i)\leq e^{-v_i}\) for all \(i\). Hence, for
\(P=\sum c_{\alpha,\beta}x^\alpha\partial^\beta\),
\[
\eta(P)\leq\max_{\alpha,\beta}|c_{\alpha,\beta}|
\eta(x)^\alpha\eta(\partial)^\beta\leq|P|_w.
\]
The converse follows by evaluating \(\eta\leq|\cdot|_w\) on the
generators. This proves the second equality.

Finally, let \(\eta\in\mathcal S(D)\) and set
\(w=\operatorname{trop}(\eta)\). Proposition~\ref{prop:gauss-section}
implies \(\eta\leq s(w)=|\cdot|_w\), so \(\eta\) belongs to the image of
\(\mathcal M(D_w)\). Since \(\eta\) was arbitrary, the union of these
images is all of \(\mathcal S(D)\).
\end{proof}

For a nonempty compact polyhedron \(\Delta\subseteq\Omega\), define
\begin{equation}\label{eq:upper-saturation}
\Delta^+=\Delta+\mathbb R_{\geq0}^{2n},
\qquad\Delta^\uparrow=\Delta^+\cap\Omega.
\end{equation}
We call \(\Delta^+\) the upper hull of \(\Delta\), and
\(\Delta^\uparrow\) its upper saturation in \(\Omega\).
The latter is again a compact polyhedron. Indeed, each coordinate on
\(\Delta^\uparrow\) is bounded below by its minimum on \(\Delta\), while
the inequalities \(u_i+v_i\leq0\) then bound the coordinates above.

\begin{theorem}\label{thm:compact-spectrum}
For every nonempty compact polyhedron \(\Delta\subseteq\Omega\),
restriction to \(D\) induces a homeomorphism
\[
\mathcal M(D_\Delta)\xrightarrow{\;\sim\;}U_{\Delta^\uparrow}.
\]
\end{theorem}

\begin{proof}
Let \(\eta\in\mathcal M(D_\Delta)\), and put
\(z=\operatorname{trop}(\eta|_D)\). For
\(m=(\alpha,\beta)\in\mathbb N^{2n}\), contractivity applied to the
ordered monomial \(x^\alpha\partial^\beta\) yields
\[
e^{-\langle z,m\rangle}\leq|x^\alpha\partial^\beta|_\Delta
=\exp\!\left(-\min_{w\in\Delta}\langle w,m\rangle\right).
\]
Hence
\begin{equation}\label{eq:upper-support-inequality}
\langle z,m\rangle\geq\min_{w\in\Delta}\langle w,m\rangle.
\end{equation}
If \(m\in\mathbb Q_{\geq0}^{2n}\), choose a positive integer \(N\) with
\(Nm\in\mathbb N^{2n}\), apply \eqref{eq:upper-support-inequality} to
\(Nm\), and divide by \(N\). By continuity in \(m\),
\eqref{eq:upper-support-inequality} therefore holds for every
\(m\in\mathbb R_{\geq0}^{2n}\).

Suppose \(z\notin\Delta^+\). Since
\(\Delta^+\) is a closed convex polyhedron, there exists \(m\in\mathbb R^{2n}\) such that
\[
\langle z,m\rangle<\inf_{y\in \Delta^+}\langle y,m\rangle.
\]
Every coordinate of \(m\) is nonnegative: if \(m_j<0\), then
\(y+te_j\in \Delta^+\) for every \(y\in \Delta^+\) and \(t\geq0\), forcing the
infimum on the right to be \(-\infty\). Thus \(m\in\mathbb R_{\geq0}^{2n}\).
Moreover,
\[
\inf_{y\in \Delta^+}\langle y,m\rangle
=\min_{w\in\Delta}\langle w,m\rangle,
\]
contradicting \eqref{eq:upper-support-inequality}. Hence \(z\in \Delta^+\), and
since \(z\in\Omega\), we obtain \(z\in\Delta^\uparrow\).

Conversely, let \(\eta\in\mathcal S(D)\) and suppose
\(z=\operatorname{trop}(\eta)\in\Delta^\uparrow\). Write \(z=w+r\) with
\(w\in\Delta\) and \(r\in\mathbb R_{\geq0}^{2n}\).
Proposition~\ref{prop:gauss-section} implies \(\eta(P)\leq|P|_z\) for
\(P\in D\). For every ordered monomial,
\[
|x^\alpha\partial^\beta|_z\leq|x^\alpha\partial^\beta|_w
\leq|x^\alpha\partial^\beta|_\Delta,
\]
and therefore \(\eta(P)\leq|P|_\Delta\) for all \(P\in D\). Hence
\(\eta\) extends uniquely to an element of \(\mathcal M(D_\Delta)\).

By density of \(D\) in \(D_\Delta\), restriction is injective. By
Section~\ref{subsec:banach-frechet-background}, \(\mathcal M(D_\Delta)\)
is compact. The resulting continuous bijection onto the Hausdorff space
\(U_{\Delta^\uparrow}\) is therefore a homeomorphism.
\end{proof}

In the toric setting of Section~\ref{subsec:toric-background}, the
polytopal algebra has bounded spectrum
\(U_\Delta^{\mathrm{com}}=\operatorname{trop}^{-1}(\Delta)\).
The difference here comes from the monomial exponent sets: Laurent monomials
on the torus range over the full character lattice, whereas the ordered
monomials of \(D\) have exponent vectors in \(\mathbb N^{2n}\). Accordingly,
boundedness for \(D_\Delta\) tests only linear forms with nonnegative
coefficient vectors, which determine the upper hull \(\Delta^+\) and hence
the upper saturation \(\Delta^\uparrow=\Delta^+\cap\Omega\), rather than
\(\Delta\) itself. In particular, \(\mathcal M(D_w)\) contains the fibers
over all \(z\in\Omega\) with \(z\geq w\), not only the fiber over \(w\).

\begin{corollary}\label{cor:proper-norming}
The tropicalization map \(\operatorname{trop}:\mathcal S(D)\to\Omega\)
is proper. Moreover, for every nonempty compact polyhedron
\(\Delta\subseteq\Omega\) and every \(f\in D_\Delta\),
\begin{equation}\label{eq:compact-norming}
|f|_\Delta=\sup_{\eta\in U_\Delta}\eta(f)
=\sup_{w\in\Delta}s(w)(f)
=\max_{q\in\operatorname{Vert}(\Delta)}s(q)(f).
\end{equation}
\end{corollary}

\begin{proof}
Let \(E\subseteq\Omega\) be compact. Choose \(t>0\) such that
\(E\subseteq\Delta_t:=\Omega\cap[-t,t]^{2n}\). Then \(\Delta_t\) is a compact
polyhedron, and Theorem~\ref{thm:compact-spectrum} identifies
\(U_{(\Delta_t)^\uparrow}\) with the compact space \(\mathcal M(D_{\Delta_t})\).
Since \(\operatorname{trop}^{-1}(E)\) is closed in \(U_{(\Delta_t)^\uparrow}\),
it is compact. Thus \(\operatorname{trop}\) is proper.

Now let \(f\in D_\Delta\). For \(w\in\Delta\),
Theorem~\ref{thm:compact-spectrum} identifies \(s(w)\) with a bounded
seminorm on \(D_\Delta\). Since its restriction to \(D\) is
\(|\cdot|_w\), continuity shows that \(s(w)(f)=|f|_w\).
Proposition~\ref{prop:series-descriptions} therefore implies
\[
\sup_{w\in\Delta}s(w)(f)
=\max_{q\in\operatorname{Vert}(\Delta)}s(q)(f)=|f|_\Delta.
\]

If \(\eta\in U_\Delta\) and \(w=\operatorname{trop}(\eta)\), then both
\(\eta\) and \(s(w)\) extend to bounded seminorms on \(D_\Delta\).
Proposition~\ref{prop:gauss-section} implies \(\eta(P)\leq s(w)(P)\) for
\(P\in D\), and by density of \(D\) in \(D_\Delta\),
\(\eta(f)\leq s(w)(f)\). Hence
\[
\sup_{\eta\in U_\Delta}\eta(f)\leq\sup_{w\in\Delta}s(w)(f).
\]
The reverse inequality follows from \(s(w)\in U_\Delta\) for
\(w\in\Delta\). This proves the asserted equalities.
\end{proof}

\subsection{Fr\'echet completions on arbitrary polyhedra}
\label{subsec:frechet-completions}

Let \(\Delta\subseteq\Omega\) be a nonempty polyhedron, and choose a
nested cofinal compact-polyhedral exhaustion
\[
C_1\subseteq C_2\subseteq\cdots\subseteq\Delta,
\qquad\bigcup_j C_j=\Delta,
\]
as in Section~\ref{subsec:polyhedral-background}. The inclusions \(C_j\subseteq C_{j+1}\) induce
continuous injective homomorphisms \(D_{C_{j+1}}\to D_{C_j}\), and hence a
projective system of Banach \(K\)-algebras. We define
\begin{equation}\label{eq:frechet-algebra-definition}
D_\Delta^{\mathrm{Fr}}=\varprojlim_jD_{C_j},
\end{equation}
endowed with the Fr\'echet topology defined by the norms
\(|\cdot|_{C_j}\), as in Section~\ref{subsec:banach-frechet-background}.

Using the common ordered-monomial realization from
Proposition~\ref{prop:series-descriptions}, the projective limit identifies
with \(\bigcap_jD_{C_j}\). Equivalently, its elements are the formal
ordered-monomial series whose weighted coefficients tend to zero uniformly
on every compact polyhedral subset of \(\Delta\).

\begin{proposition}\label{prop:frechet-completion}
The Fr\'echet algebra \(D_\Delta^{\mathrm{Fr}}\) is complete and, as a
topological \(K\)-algebra, independent of the chosen cofinal compact-polyhedral
exhaustion. If \(\Delta\) is compact, then
\(D_\Delta^{\mathrm{Fr}}=D_\Delta\). If \(\Delta'\subseteq\Delta\) are
nonempty polyhedra, there is a natural continuous injective homomorphism
\[
D_\Delta^{\mathrm{Fr}}\longrightarrow D_{\Delta'}^{\mathrm{Fr}}.
\]
\end{proposition}

\begin{proof}
A sequence in \(D_\Delta^{\mathrm{Fr}}\) is Cauchy precisely when it is
Cauchy in every \(D_{C_j}\). Since each \(D_{C_j}\) is Banach, the
componentwise limits exist and are compatible, so \(D_\Delta^{\mathrm{Fr}}\)
is complete.

If \((C_j)\) and \((C'_k)\) are two cofinal compact-polyhedral exhaustions
of \(\Delta\), then every \(C_j\) is contained in some \(C'_k\), and
conversely. Hence the two families of norms are cofinal and define the same
algebra and topology.

If \(\Delta\) is compact, cofinality applied to \(\Delta\) itself shows
that \(C_j=\Delta\) for some \(j\), and therefore
\(D_\Delta^{\mathrm{Fr}}=D_\Delta\).

Finally, let \(\Delta'\subseteq\Delta\). Under the ordered-monomial
description above, every series in \(D_\Delta^{\mathrm{Fr}}\) converges
uniformly on each compact polyhedron contained in \(\Delta\), and hence on
each compact polyhedron contained in \(\Delta'\). Thus
\[
D_\Delta^{\mathrm{Fr}}\subseteq D_{\Delta'}^{\mathrm{Fr}}
\]
inside the common space of formal ordered-monomial series. The inclusion is
continuous because each defining seminorm for \(D_{\Delta'}^{\mathrm{Fr}}\)
is dominated by one of the defining seminorms for \(D_\Delta^{\mathrm{Fr}}\).
\end{proof}

\begin{remark}\label{rem:toric-unbounded-comparison}
As recalled in Section~\ref{subsec:toric-background}, the toric
construction uses the recession cone to choose an ambient toric chart.
Our Fr\'echet construction keeps \(\mathcal S(D)\) fixed and uses compact
exhaustion.
\end{remark}

\section{Rational models, flatness, and base change}\label{sec:flatness}

The constructions of Section~\ref{sec:polyhedral-completions} are defined
for arbitrary real polyhedra and require no rationality assumption. We now
restrict to the rational setting in order to obtain finite-type models over
the valuation ring and derive flatness from their adic completions.

Assume throughout this section that \(K\) is discretely valued, with
valuation ring \(R=K^\circ\), uniformizer \(\pi\), residue field \(k\),
and \(\operatorname{val}(\pi)=1\). Thus
\(\Gamma=\operatorname{val}(K^*)=\mathbb Z\).

\subsection{Models over the valuation ring}

Let \(\Delta\subseteq\Omega\) be a nonempty compact
\(\Gamma\)-rational polyhedron. For
\(m=(\alpha,\beta)\in\mathbb N^{2n}\), set
\[
\lambda_\Delta(m)=\min_{w\in\Delta}\langle w,m\rangle,
\qquad a_\Delta(m)=\left\lceil-\lambda_\Delta(m)\right\rceil.
\]
Define
\[
D_\Delta^\circ=\{f\in D_\Delta:|f|_\Delta\leq1\},
\qquad A_\Delta=D\cap D_\Delta^\circ.
\]
For \(P=\sum_{\alpha,\beta}c_{\alpha,\beta}x^\alpha\partial^\beta\in D\),
Proposition~\ref{prop:series-descriptions} shows that
\[
P\in A_\Delta\quad\Longleftrightarrow\quad
\operatorname{val}(c_{\alpha,\beta})+\lambda_\Delta(\alpha,\beta)\geq0
\quad\text{for every }(\alpha,\beta).
\]
Since the coefficient valuations are integral, this is equivalent to
\(\operatorname{val}(c_{\alpha,\beta})\geq a_\Delta(\alpha,\beta)\).
When the vertices of \(\Delta\) are integral,
\(a_\Delta(m)=-\lambda_\Delta(m)\); the rounding above is needed only for
rational polyhedra.
This coefficient condition is the counterpart of the condition defining
\(K[M]^\Delta\) in Section~\ref{subsec:toric-background}, with ordered monomials \(x^\alpha\partial^\beta\) indexed by \(\mathbb N^{2n}\).

\begin{proposition}\label{prop:integral-model}
The algebra \(A_\Delta\) is finitely generated as an \(R\)-algebra, left
and right Noetherian, and flat as an \(R\)-module.
\end{proposition}

\begin{proof}
Equip \(A_\Delta\) with the filtration induced by total ordered-monomial
degree on \(D\). Since the commutator terms have smaller degree, the
associated graded ring is the commutative \(R\)-algebra
\[
\operatorname{gr}A_\Delta
=\bigoplus_{m=(\alpha,\beta)\in\mathbb N^{2n}}
\pi^{a_\Delta(m)}R\,x^\alpha\xi^\beta.
\]
Let \(w^{(1)},\ldots,w^{(r)}\) be the vertices of \(\Delta\), and consider
\[
S_\Delta=\left\{(q,m)\in\mathbb Z\times\mathbb N^{2n}:
q+\langle w^{(j)},m\rangle\geq0\text{ for }1\leq j\leq r\right\}.
\]
This is the semigroup of integral points of a rational polyhedral cone,
hence is finitely generated by Gordan's lemma. For integral \(q\), the
defining inequalities are equivalent to \(q\geq a_\Delta(m)\). Thus
\(\operatorname{gr}A_\Delta\) is the \(R\)-subalgebra of
\(K[x_1,\ldots,x_n,\xi_1,\ldots,\xi_n]\) generated by the monomials
\(\pi^qx^\alpha\xi^\beta\) with \((q,\alpha,\beta)\in S_\Delta\).
In particular, \(\operatorname{gr}A_\Delta\) is a finitely generated
Noetherian \(R\)-algebra.

Lifting a finite set of homogeneous algebra generators of
\(\operatorname{gr}A_\Delta\) and inducting on the filtration degree shows
that \(A_\Delta\) is finitely generated over \(R\). Since the total-degree
filtration is exhaustive and discrete and \(\operatorname{gr}A_\Delta\)
is Noetherian, the standard filtered-ring theorem
\cite[Chapter~I, Corollary~5.7]{NVO79}, applied also to the opposite ring,
implies that \(A_\Delta\) is left and right Noetherian. Finally,
\(A_\Delta\subseteq D\) is torsion-free over the discrete valuation ring
\(R\), hence \(R\)-flat.
\end{proof}

Let \(\widehat A_\Delta=\varprojlim_{\ell\geq1} A_\Delta/\pi^{\ell}A_\Delta\).

\begin{proposition}\label{prop:integral-completion}
There are canonical isomorphisms
\begin{equation}\label{eq:model-completion}
\widehat A_\Delta\simeq D_\Delta^\circ,\qquad
A_\Delta[1/\pi]\simeq D,\qquad
\widehat A_\Delta[1/\pi]\simeq D_\Delta.
\end{equation}
\end{proposition}

\begin{proof}
Since \(D_\Delta^\circ\) is the closed unit ball of the Banach algebra
\(D_\Delta\), it is complete. For every \(\ell\geq0\), we have
\[
\pi^{\ell}D_\Delta^\circ
=\{f\in D_\Delta^\circ:|f|_\Delta\leq|\pi|^\ell\}.
\]
Hence the \(\pi\)-adic topology on \(D_\Delta^\circ\) agrees with the
topology induced by \(|\cdot|_\Delta\). Since \(D_\Delta^\circ\) is
complete for the latter topology, it is therefore \(\pi\)-adically
complete.

By Proposition~\ref{prop:series-descriptions}, the finite ordered-monomial
truncations of every \(f\in D_\Delta^\circ\) converge to \(f\) in
\(D_\Delta\). Truncation does not increase the \(\Delta\)-norm, so these
truncations belong to \(A_\Delta\). Hence \(A_\Delta\) is dense in
\(D_\Delta^\circ\).

Moreover, \(A_\Delta\cap\pi^{\ell}D_\Delta^\circ=\pi^{\ell}A_\Delta\).
Only the inclusion from left to right requires verification. If
\(P\in A_\Delta\cap\pi^{\ell}D_\Delta^\circ\), then \(\pi^{-\ell}P\) lies
in both \(D\) and \(D_\Delta^\circ\), hence in \(A_\Delta\).
It follows that the inclusion \(A_\Delta\subseteq D_\Delta^\circ\)
induces, for every \(\ell\), an isomorphism
\[
A_\Delta/\pi^{\ell}A_\Delta\xrightarrow{\sim}
D_\Delta^\circ/\pi^{\ell}D_\Delta^\circ.
\]
Injectivity follows from the preceding intersection identity, while
surjectivity follows from density. Since \(D_\Delta^\circ\) is
\(\pi\)-adically complete, we deduce that
\(\widehat A_\Delta\simeq D_\Delta^\circ\).

Finally, every element of \(D\) becomes an element of \(A_\Delta\) after
multiplication by a sufficiently large power of \(\pi\), so
\(A_\Delta[1/\pi]=D\). Likewise, every element of \(D_\Delta\) becomes
an element of \(D_\Delta^\circ\) after multiplication by a sufficiently
large power of \(\pi\). Consequently,
\(\widehat A_\Delta[1/\pi]\simeq D_\Delta^\circ[1/\pi]=D_\Delta\).
\end{proof}

\subsection{Compact rational flatness}

The finite-type models above provide the flatness needed later for base
change.

\begin{theorem}\label{thm:compact-flat}
Let \(\Delta\subseteq\Omega\) be a nonempty compact \(\Gamma\)-rational
polyhedron. Then the canonical homomorphism \(D\to D_\Delta\) is flat on
both sides. Moreover, \(\pi\)-adic completion is exact on finitely generated
left \(A_\Delta\)-modules, and for every such module \(M\) there are
canonical isomorphisms
\[
\widehat M\simeq\widehat A_\Delta\otimes_{A_\Delta}M
\simeq D_\Delta^\circ\otimes_{A_\Delta}M,
\]
where \(\widehat M\) denotes the \(\pi\)-adic completion of \(M\).
\end{theorem}

\begin{proof}
Since \(A_\Delta\) is left and right Noetherian and \(\pi\) is central,
its \(\pi\)-adic completion \(\widehat A_\Delta\) is flat over
\(A_\Delta\) on both sides; see
\cite[Chapter~I, Propositions~5.10 and~5.12]{NVO79}.
Proposition~\ref{prop:integral-completion} identifies
\(A_\Delta[1/\pi]=D\) and \(\widehat A_\Delta[1/\pi]=D_\Delta\).
Localizing at the central element \(\pi\) preserves flatness, so
\(D\to D_\Delta\) is flat on both sides. The assertions about completion of
finite modules follow from the same cited results and
Proposition~\ref{prop:integral-completion}.
\end{proof}

\begin{proposition}\label{prop:compact-relations}
Let \(N\subseteq D^r\) be a finitely generated left submodule. Then
\(D_\Delta N=\overline N\) in \(D_\Delta^r\). In particular,
\(D_\Delta N\) is closed.
\end{proposition}

\begin{proof}
Let \(L=N\cap A_\Delta^r\). Since \(A_\Delta\) is left Noetherian,
\(L\) is finitely generated. Moreover, \(L[1/\pi]=N\). Indeed, if
\(x\in N\), then \(\pi^mx\in A_\Delta^r\) for some \(m\), and hence
\(\pi^mx\in L\); the reverse inclusion is immediate. Since
\(D_\Delta=D_\Delta^\circ[1/\pi]\), it follows that
\begin{equation}\label{eq:localized-relations}
D_\Delta N=D_\Delta L=(D_\Delta^\circ L)[1/\pi].
\end{equation}

The quotient \(A_\Delta^r/L\) is \(\pi\)-torsion-free. If
\(x\in A_\Delta^r\) and \(\pi x\in L\), then \(\pi x\in N\); since
\(\pi\) is invertible in \(D\), it follows that \(x\in N\), and hence
\(x\in L\).
Exactness of \(\pi\)-adic completion yields
\[
0\longrightarrow D_\Delta^\circ L
\longrightarrow(D_\Delta^\circ)^r
\longrightarrow\widehat{A_\Delta^r/L}\longrightarrow0.
\]
Since \(\widehat{A_\Delta^r/L}\) is \(\pi\)-adically separated,
\(D_\Delta^\circ L\) is closed in \((D_\Delta^\circ)^r\) for the
\(\pi\)-adic topology.

Moreover,
\[
\widehat{A_\Delta^r/L}\simeq(D_\Delta^\circ)^r/D_\Delta^\circ L
\simeq D_\Delta^\circ\otimes_{A_\Delta}(A_\Delta^r/L).
\]
Multiplication by \(\pi\) is injective on \(A_\Delta^r/L\), and
flatness of \(D_\Delta^\circ\) over \(A_\Delta\) preserves this
injection after tensoring. Thus \((D_\Delta^\circ)^r/D_\Delta^\circ L\)
is \(\pi\)-torsion-free. Equivalently,
\begin{equation}\label{eq:saturated-relations}
(D_\Delta^\circ L)[1/\pi]\cap(D_\Delta^\circ)^r=D_\Delta^\circ L.
\end{equation}

We now compare the two topologies. The module \(D_\Delta^r\) carries the
Banach topology defined by the maximum \(\Delta\)-norm, while
\((D_\Delta^\circ)^r\) carries its \(\pi\)-adic topology. By
Proposition~\ref{prop:integral-completion}, the latter agrees with the
topology induced from \(D_\Delta^r\).

Suppose that \(y_j\to y\) in \(D_\Delta^r\), with \(y_j\in D_\Delta N\).
Then there is \(C>0\) such that
\[
\|y_j\|_\Delta,\ \|y\|_\Delta\leq C\qquad\text{for all }j.
\]
Choose \(m\) such that \(|\pi|^mC\leq1\). Then
\(\pi^my_j,\pi^my\in(D_\Delta^\circ)^r\). By
\eqref{eq:localized-relations}, \(y_j\in(D_\Delta^\circ L)[1/\pi]\),
so \eqref{eq:saturated-relations} implies that
\(\pi^my_j\in D_\Delta^\circ L\). Since \(\pi^my_j\to\pi^my\) in
the induced norm topology, and hence in the \(\pi\)-adic topology on
\((D_\Delta^\circ)^r\), closedness of \(D_\Delta^\circ L\) implies
that \(\pi^my\in D_\Delta^\circ L\). By
\eqref{eq:localized-relations} again, \(y\in D_\Delta N\). Thus
\(D_\Delta N\) is closed.

Finally, let \(g_1,\ldots,g_s\) generate \(N\) over \(D\). Since
\(D\) is dense in \(D_\Delta\), every element of
\(D_\Delta N=\sum_iD_\Delta g_i\) is a limit of elements of
\(N=\sum_iDg_i\). Hence \(N\) is dense in \(D_\Delta N\), and
therefore \(D_\Delta N=\overline N\).
\end{proof}

\subsection{Flatness for arbitrary rational polyhedra}

Let \(\Delta\subseteq\Omega\) be a nonempty \(\Gamma\)-rational
polyhedron. Choose, as in Section~\ref{subsec:polyhedral-background},
a nested cofinal exhaustion
\(C_1\subseteq C_2\subseteq\cdots\subseteq\Delta\) by nonempty compact
\(\Gamma\)-rational polyhedra, so that
\(D_\Delta^{\mathrm{Fr}}=\varprojlim_jD_{C_j}\).

We first record an inverse-limit lemma.

\begin{lemma}\label{lem:inverse-limit}
Suppose that
\[
0\longrightarrow K_j\longrightarrow E_j\longrightarrow F_j\longrightarrow0
\]
is a compatible projective system of exact sequences of nonarchimedean
Banach \(K\)-spaces and continuous linear maps, with each \(K_j\)
carrying the subspace norm. Assume that the transition maps
\(E_{j+1}\to E_j\) are contractive and that \(K_{j+1}\to K_j\) has dense
image for every \(j\). Then
\[
0\longrightarrow\varprojlim_jK_j\longrightarrow\varprojlim_jE_j
\longrightarrow\varprojlim_jF_j\longrightarrow0
\]
is exact.
\end{lemma}

\begin{proof}
Let \((f_j)\in\varprojlim_jF_j\). Choose \(e_1\in E_1\) mapping to
\(f_1\). Inductively, suppose that \(e_j\) has been chosen, and choose
any lift \(e'_{j+1}\in E_{j+1}\) of \(f_{j+1}\). The difference between
\(e_j\) and the image of \(e'_{j+1}\) in \(E_j\) lies in \(K_j\).
By density, we may choose \(k_{j+1}\in K_{j+1}\) such that, after replacing
\(e'_{j+1}\) by \(e_{j+1}=e'_{j+1}+k_{j+1}\), the norm of this
difference is less than \(|\pi|^j\).

For each fixed \(i\), the images of \(e_j\) in \(E_i\), for \(j\geq i\),
form a Cauchy sequence. The transition maps are contractive and the
successive differences tend to zero. Let \(e_i^\infty\) denote its limit.
The elements \(e_i^\infty\) are compatible, and continuity shows that
\(e_i^\infty\) maps to \(f_i\). Thus \((f_j)\) lifts to an element of
\(\varprojlim_jE_j\).
\end{proof}

The compact result now passes to the Fr\'echet completion.

\begin{proposition}\label{prop:frechet-relations}
Let \(N\subseteq D^r\) be a finitely generated left submodule. Then
\begin{equation}\label{eq:relation-limit}
D_\Delta^{\mathrm{Fr}}N
=\varprojlim_jD_{C_j}N=\overline N,
\end{equation}
where the closure is taken in \((D_\Delta^{\mathrm{Fr}})^r\). In
particular, \(D_\Delta^{\mathrm{Fr}}N\) is closed.
\end{proposition}

\begin{proof}
Choose generators of \(N\), and let \(\varphi:D^s\to D^r\) be the
resulting homomorphism with image \(N\). Set \(N_1=\ker\varphi\).
Since \(D\) is left Noetherian, \(N_1\) is finitely generated.
For every \(j\), Theorem~\ref{thm:compact-flat} yields an exact sequence
\[
0\longrightarrow D_{C_j}N_1\longrightarrow D_{C_j}^s
\longrightarrow D_{C_j}N\longrightarrow0.
\]
The transition maps \(D_{C_{j+1}}^s\to D_{C_j}^s\) are contractive
because \(|\cdot|_{C_j}\leq|\cdot|_{C_{j+1}}\). By
Proposition~\ref{prop:compact-relations}, \(D_{C_j}N_1\) is the closure
of \(N_1\) in \(D_{C_j}^s\). Hence
\(D_{C_{j+1}}N_1\to D_{C_j}N_1\) has dense image, since its image
contains \(N_1\). Lemma~\ref{lem:inverse-limit} therefore shows that the
natural map
\[
(D_\Delta^{\mathrm{Fr}})^s=\varprojlim_jD_{C_j}^s
\longrightarrow\varprojlim_jD_{C_j}N
\]
is surjective. Its image is \(D_\Delta^{\mathrm{Fr}}N\), proving the
first equality in \eqref{eq:relation-limit}.

Each \(D_{C_j}N\) is closed in \(D_{C_j}^r\) by
Proposition~\ref{prop:compact-relations}. If
\(p_j:(D_\Delta^{\mathrm{Fr}})^r\to D_{C_j}^r\) denotes the canonical
projection, then
\[
\varprojlim_jD_{C_j}N=\bigcap_jp_j^{-1}(D_{C_j}N),
\]
so this submodule is closed in \((D_\Delta^{\mathrm{Fr}})^r\).

Finite ordered-monomial truncations show that \(D\) is dense in
\(D_\Delta^{\mathrm{Fr}}\). If \(g_1,\ldots,g_s\) generate \(N\), then
\(N=\sum_iDg_i\) is dense in
\(D_\Delta^{\mathrm{Fr}}N=\sum_iD_\Delta^{\mathrm{Fr}}g_i\). Since
\(D_\Delta^{\mathrm{Fr}}N\) is closed, the second equality in
\eqref{eq:relation-limit} follows.
\end{proof}

The same argument applies to finitely generated right submodules.

\begin{theorem}\label{thm:frechet-flat}
For every nonempty \(\Gamma\)-rational polyhedron \(\Delta\subseteq\Omega\),
the canonical homomorphism \(D\to D_\Delta^{\mathrm{Fr}}\) is flat on both
sides.
\end{theorem}

\begin{proof}
We first prove right flatness. Let \(I\subseteq D\) be a left ideal. Since
\(D\) is left Noetherian, choose a finite presentation
\[
D^t\xrightarrow{d_2}D^s\xrightarrow{d_1}I\longrightarrow0,
\]
and regard \(d_1\) also as a map \(D^s\to D\). Set
\(N=\ker d_1=\operatorname{im}d_2\).
For every \(j\), Theorem~\ref{thm:compact-flat} implies that
\[
\ker\!\left(D_{C_j}^s\longrightarrow D_{C_j}\right)=D_{C_j}N.
\]
Taking projective limits and applying
Proposition~\ref{prop:frechet-relations}, we obtain
\[
\begin{split}
\ker\!\left((D_\Delta^{\mathrm{Fr}})^s\longrightarrow D_\Delta^{\mathrm{Fr}}\right)
&=D_\Delta^{\mathrm{Fr}}N\\
&=\operatorname{im}\!\left((D_\Delta^{\mathrm{Fr}})^t
\longrightarrow(D_\Delta^{\mathrm{Fr}})^s\right).
\end{split}
\]
Thus the sequence obtained from the presentation of \(I\) is exact at
\((D_\Delta^{\mathrm{Fr}})^s\). Consequently, the canonical map
\(D_\Delta^{\mathrm{Fr}}\otimes_DI\to D_\Delta^{\mathrm{Fr}}\) is
injective. The left-ideal criterion for flatness shows that
\(D_\Delta^{\mathrm{Fr}}\) is flat as a right \(D\)-module.

Applying the same argument to right ideals, using the right-module version
of Proposition~\ref{prop:frechet-relations}, proves left flatness.
\end{proof}

\subsection{Base change of finitely generated \texorpdfstring{\(D\)}{D}-modules}
\label{subsec:finite-base-change}

Let \(\Delta\subseteq\Omega\) be a nonempty \(\Gamma\)-rational
polyhedron, and let \(M\) be a finitely generated left \(D\)-module.
Choose a finite presentation \(M\simeq D^r/N\). By right exactness of
tensor product,
\[
D_\Delta^{\mathrm{Fr}}\otimes_DM
\simeq(D_\Delta^{\mathrm{Fr}})^r/D_\Delta^{\mathrm{Fr}}N.
\]
Proposition~\ref{prop:frechet-relations} identifies
\(D_\Delta^{\mathrm{Fr}}N=\overline N\); hence
\begin{equation}\label{eq:frechet-presentation}
D_\Delta^{\mathrm{Fr}}\otimes_DM\simeq
(D_\Delta^{\mathrm{Fr}})^r/\overline N.
\end{equation}
In particular, this is a Hausdorff Fr\'echet
\(D_\Delta^{\mathrm{Fr}}\)-module. If \(\Delta\) is compact, it is a
Banach \(D_\Delta\)-module. This topology is independent of the finite
presentation. Indeed, lift the identity of \(M\) in both directions
between two finite free presentations. The resulting matrices over \(D\)
extend continuously after base change, and induce mutually inverse
continuous maps on the quotient modules.

The presentation is also compatible with compact exhaustion.

\begin{proposition}\label{prop:finite-module-limit}
Let \(C_1\subseteq C_2\subseteq\cdots\subseteq\Delta\) be a nested
cofinal exhaustion by nonempty compact \(\Gamma\)-rational polyhedra,
as in Section~\ref{subsec:polyhedral-background}. Then the canonical map
\begin{equation}\label{eq:finite-module-limit}
D_\Delta^{\mathrm{Fr}}\otimes_DM
\longrightarrow\varprojlim_j(D_{C_j}\otimes_DM)
\end{equation}
is an isomorphism of topological \(D_\Delta^{\mathrm{Fr}}\)-modules.
\end{proposition}

\begin{proof}
For every \(j\), Theorem~\ref{thm:compact-flat} and
Proposition~\ref{prop:compact-relations} yield an exact sequence of Banach
spaces
\[
0\longrightarrow D_{C_j}N\longrightarrow D_{C_j}^r
\longrightarrow D_{C_j}\otimes_DM\longrightarrow0.
\]
The transition maps \(D_{C_{j+1}}N\to D_{C_j}N\) have dense image, since
their images contain \(N\), which is dense in \(D_{C_j}N\).
Lemma~\ref{lem:inverse-limit} therefore yields an exact sequence
\[
0\longrightarrow\varprojlim_jD_{C_j}N
\longrightarrow(D_\Delta^{\mathrm{Fr}})^r
\longrightarrow\varprojlim_j(D_{C_j}\otimes_DM)\longrightarrow0.
\]
By Proposition~\ref{prop:frechet-relations}, the first term is
\(D_\Delta^{\mathrm{Fr}}N\). Together with
\eqref{eq:frechet-presentation}, this proves that
\eqref{eq:finite-module-limit} is an algebraic isomorphism.

It remains to compare the topologies. The quotient topology on
\((D_\Delta^{\mathrm{Fr}})^r/D_\Delta^{\mathrm{Fr}}N\) is defined by the
seminorms
\[
q_j(\bar x)=\inf_{y\in D_\Delta^{\mathrm{Fr}}N}\|x+y\|_{C_j}.
\]
Let \(x_j\) denote the image of \(x\) in \(D_{C_j}^r\). The image of
\(D_\Delta^{\mathrm{Fr}}N\) in \(D_{C_j}^r\) contains \(N\), and is
therefore dense in \(D_{C_j}N\). Consequently,
\[
q_j(\bar x)=\inf_{y\in D_{C_j}N}\|x_j+y\|_{C_j},
\]
which is the quotient norm on
\(D_{C_j}^r/D_{C_j}N\simeq D_{C_j}\otimes_DM\). Thus the quotient
topology in \eqref{eq:frechet-presentation} agrees with the projective-limit
topology in \eqref{eq:finite-module-limit}.
\end{proof}

The canonical image of \(M\) in \(D_\Delta^{\mathrm{Fr}}\otimes_DM\) is
dense, since \(D^r\) is dense in \((D_\Delta^{\mathrm{Fr}})^r\). It need
not be injective.

\begin{corollary}\label{cor:exact-base-change}
For every nonempty \(\Gamma\)-rational polyhedron
\(\Delta\subseteq\Omega\), the functor
\(D_\Delta^{\mathrm{Fr}}\otimes_D-\) is exact on left \(D\)-modules.
For finitely generated \(D\)-modules, its values are Hausdorff Fr\'echet
modules and satisfy the projective-limit description
\eqref{eq:finite-module-limit}.
\end{corollary}

\begin{proof}
Exactness follows from Theorem~\ref{thm:frechet-flat}. The remaining
assertions follow from \eqref{eq:frechet-presentation} and
Proposition~\ref{prop:finite-module-limit}.
\end{proof}

\section{Toric-face degeneration and reduction}\label{sec:integral-models}

Section~\ref{sec:flatness} associates a finite-type \(R\)-model
\(A_\Delta\) with every compact \(\Gamma\)-rational polyhedron. Recall
that, for \(m\in\mathbb N^{2n}\),
\[
\lambda_\Delta(m)=\min_{w\in\Delta}\langle w,m\rangle,
\qquad a_\Delta(m)=\left\lceil-\lambda_\Delta(m)\right\rceil.
\]
We now impose the stronger hypotheses that \(\Delta\) has integral vertices
and lies in \(\Omega^\circ\). Under these assumptions,
\(\lambda_\Delta(m)\in\mathbb Z\) and
\(a_\Delta(m)=-\lambda_\Delta(m)\), and the special fiber of
\(A_\Delta\) is described by the normal fan of its upper hull.

Throughout this section, \(K\) is complete and discretely valued, with
\(R\), \(\pi\), and \(k\) as in Section~\ref{sec:flatness}, and
\(\Delta\subseteq\Omega^\circ\) is a nonempty compact polyhedron with
integral vertices.

\subsection{The toric-face special fiber}

Recall from \eqref{eq:upper-saturation} that
\(\Delta^+=\Delta+\mathbb R_{\geq0}^{2n}\). Let
\(\Sigma_{\Delta^+}\) be its minimum normal fan. Since
\(\operatorname{rec}(\Delta^+)=\mathbb R_{\geq0}^{2n}\),
Section~\ref{subsec:polyhedral-background} implies
\(|\Sigma_{\Delta^+}|=\mathbb R_{\geq0}^{2n}\). For a nonempty face
\(F\) of \(\Delta^+\) and \(w\in\operatorname{relint}(F)\), its
normal cone is
\[
N_{\Delta^+}(F)=\{m\in\mathbb R_{\geq0}^{2n}:
\langle w,m\rangle=\lambda_\Delta(m)\}.
\]
Here \(\lambda_\Delta(m)=\min_{w\in\Delta}\langle w,m\rangle\) is
understood on all of \(\mathbb R_{\geq0}^{2n}\).

For \(m=(\alpha,\beta)\in\mathbb N^{2n}\), put
\[
e_m=\pi^{-\lambda_\Delta(m)}x^\alpha\partial^\beta\in A_\Delta.
\]
Let \(k[\Sigma_{\Delta^+}]\) denote the associated toric face ring. It has
\(k\)-basis \(z^m\), \(m\in\mathbb N^{2n}\), with multiplication
\[
z^mz^{m'}=\begin{cases}
z^{m+m'},&\text{if }m,m'\text{ lie in a common cone of }\Sigma_{\Delta^+},\\
0,&\text{otherwise}.
\end{cases}
\]
See \cite[Section~1]{IchimRomer} for toric face rings.

\begin{theorem}\label{thm:toric-face-reduction}
There are canonical isomorphisms
\[
A_\Delta/\pi A_\Delta\simeq D_\Delta^\circ/\pi D_\Delta^\circ
\simeq k[\Sigma_{\Delta^+}],
\]
under which \(z^m\) corresponds to the reduction \(\overline e_m\).
\end{theorem}

\begin{proof}
Proposition~\ref{prop:integral-completion} identifies \(D_\Delta^\circ\)
with the \(\pi\)-adic completion of \(A_\Delta\), and therefore
\(A_\Delta/\pi A_\Delta\simeq D_\Delta^\circ/\pi D_\Delta^\circ\).
The coefficient description of \(A_\Delta\) shows that the \(e_m\) form
an \(R\)-basis of \(A_\Delta\). Their reductions \(\overline e_m\)
therefore form a \(k\)-basis of \(A_\Delta/\pi A_\Delta\).

For \(m,m'\in\mathbb N^{2n}\), the definition of \(\lambda_\Delta\)
implies \(\lambda_\Delta(m+m')\geq\lambda_\Delta(m)+\lambda_\Delta(m')\).
The term of exponent \(m+m'\) in \(e_me_{m'}\) consequently has
coefficient
\[
\pi^{\lambda_\Delta(m+m')-\lambda_\Delta(m)-\lambda_\Delta(m')}
\]
with respect to \(e_{m+m'}\). Its reduction is \(1\) precisely when
equality holds, equivalently when \(m\) and \(m'\) lie in a common cone
of \(\Sigma_{\Delta^+}\); otherwise it reduces to \(0\).

Every other ordered monomial occurring in
\(x^\alpha\partial^\beta x^{\alpha'}\partial^{\beta'}\) has exponent
\(m''=m+m'-(q,q)\), where \(q\in\mathbb N^n\) is nonzero.
Since \(\Delta\subseteq\Omega^\circ\) and its vertices are integral,
every vertex \(w\) satisfies \(w_i+w_{n+i}\leq-1\). Hence
\[
\lambda_\Delta(m'')\geq\lambda_\Delta(m+m')+1
\geq\lambda_\Delta(m)+\lambda_\Delta(m')+1.
\]
The coefficients in the ordered product are integers, so their absolute
values are at most one. Thus every commutator correction is divisible by
\(\pi\) after normalization and vanishes in the special fiber.
Consequently,
\[
\overline e_m\,\overline e_{m'}=\begin{cases}
\overline e_{m+m'},&\text{if }m,m'\text{ lie in a common cone of }\Sigma_{\Delta^+},\\
0,&\text{otherwise}.
\end{cases}
\]
Hence \(z^m\mapsto\overline e_m\) defines the asserted isomorphism.
\end{proof}

\begin{remark}\label{rem:rational-special-fiber}
The rational model of Section~\ref{sec:flatness} need not have a toric-face
special fiber. For example, take \(n=1\) and
\(\Delta=\{(-1/2,-1/2)\}\). Then \(\Sigma_{\Delta^+}\) has a single
maximal cone, whereas in \(A_\Delta/\pi A_\Delta\) the nonzero classes
of \(\pi x\), \(\pi\partial\), and \(\pi x\partial\) satisfy
\[
\overline{\pi x}\,\overline{\pi\partial}=0,
\qquad\overline{\pi x\partial}\neq0.
\]
Thus the rounding in \(a_\Delta(m)=\lceil-\lambda_\Delta(m)\rceil\)
changes the multiplication in the special fiber.
\end{remark}

\subsection{Reduction and toric strata}

For a cone \(\sigma\in\Sigma_{\Delta^+}\), let
\(\mathfrak p_\sigma\subseteq k[\Sigma_{\Delta^+}]\) be the monomial prime
generated by the \(z^m\) with \(m\notin\sigma\).

Following the construction recalled in
Section~\ref{subsec:toric-background}, we define a reduction map.
For \(\eta\in\mathcal M(D_\Delta)\), the set
\(\{f\in D_\Delta^\circ:\eta(f)<1\}\) is a prime ideal of
\(D_\Delta^\circ\) containing \(\pi D_\Delta^\circ\). By
Theorem~\ref{thm:toric-face-reduction}, it determines a prime of
\(k[\Sigma_{\Delta^+}]\). We define
\[
\operatorname{red}_\Delta:\mathcal M(D_\Delta)
\longrightarrow\operatorname{Spec}k[\Sigma_{\Delta^+}],
\qquad\operatorname{red}_\Delta(\eta)=\{\overline f:\eta(f)<1\}.
\]

\begin{proposition}\label{prop:reduction-anticontinuous}
The map \(\operatorname{red}_\Delta\) is anticontinuous.
\end{proposition}

\begin{proof}
Let \(V(I)\) be a closed subset of the special fiber. Since
\(k[\Sigma_{\Delta^+}]\) is Noetherian, write
\(I=(\overline f_1,\ldots,\overline f_s)\), with
\(f_i\in D_\Delta^\circ\). Then
\[
\operatorname{red}_\Delta^{-1}(V(I))=
\bigcap_{i=1}^s\{\eta\in\mathcal M(D_\Delta):\eta(f_i)<1\},
\]
which is open.
\end{proof}

Recall from Theorem~\ref{thm:compact-spectrum} that
\(\mathcal M(D_\Delta)\simeq U_{\Delta^\uparrow}\), where
\(\Delta^\uparrow=\Delta^+\cap\Omega\).
Thus, if \(\eta\in\mathcal M(D_\Delta)\) and
\(w=\operatorname{trop}(\eta)\), then \(w\in\Delta^\uparrow\).
For such \(w\), set
\begin{equation}\label{eq:normal-cone-at-weight}
\sigma_w=\{m\in\mathbb R_{\geq0}^{2n}:
\langle w,m\rangle=\lambda_\Delta(m)\}.
\end{equation}
If \(F\) is the unique face of \(\Delta^+\)
whose relative interior contains \(w\), then
\(\sigma_w=N_{\Delta^+}(F)\) by Section~\ref{subsec:polyhedral-background}.
We also write \(\sigma_F=N_{\Delta^+}(F)\).

Recall also that the Gauss section \(s:\Omega\to\mathcal S(D)\) is
defined on \(D\) by
\[
s(w)\!\left(\sum_{\alpha,\beta}c_{\alpha,\beta}
x^\alpha\partial^\beta\right)
=\max_{\alpha,\beta}|c_{\alpha,\beta}|
e^{-\langle w,(\alpha,\beta)\rangle}.
\]
For \(w\in\Delta^\uparrow\), it extends to a point of
\(\mathcal M(D_\Delta)\).

\begin{samepage}
\begin{theorem}\label{thm:reduction-map}
Let \(\eta\in\mathcal M(D_\Delta)\) and put
\(w=\operatorname{trop}(\eta)\).
\begin{enumerate}[label=\textup{(\roman*)}]
\item A monomial \(z^m\) belongs to \(\operatorname{red}_\Delta(\eta)\)
if and only if \(m\notin\sigma_w\).
\item For the Gauss point \(s(w)\), one has
\(\operatorname{red}_\Delta(s(w))=\mathfrak p_{\sigma_w}\).
\end{enumerate}
\end{theorem}
\end{samepage}

\begin{proof}
For \(m=(\alpha,\beta)\in\mathbb N^{2n}\), we have
\[
\eta(e_m)=e^{\lambda_\Delta(m)-\langle w,m\rangle}.
\]
Since \(w\in\Delta^\uparrow\), we have
\(\langle w,m\rangle\geq\lambda_\Delta(m)\). Thus
\(\eta(e_m)<1\) precisely when \(m\notin\sigma_w\). Since \(e_m\)
reduces to \(z^m\), this proves (i).

For (ii), let \(f\in D_\Delta^\circ\). Write \(f=\sum_m b_me_m\),
where \(b_m\in R\) and \(|b_m|\to0\) as
\(|\alpha|+|\beta|\to\infty\), with \(m=(\alpha,\beta)\). Then
\(\overline f=\sum_{|b_m|=1}\overline b_mz^m\), and the Gauss
formula implies
\[
s(w)(f)=\sup_m|b_m|e^{\lambda_\Delta(m)-\langle w,m\rangle}.
\]
The exponential factor is at most \(1\), with equality precisely for
\(m\in\sigma_w\). Indeed, if a monomial occurring in \(\overline f\)
has exponent in \(\sigma_w\), then \(s(w)(f)=1\). Conversely, if no
such monomial has exponent in \(\sigma_w\), the finitely many terms
with \(|b_m|=1\) have value strictly less than \(1\), while all remaining
terms have value at most \(e^{-1}<1\). Hence \(s(w)(f)<1\).
Therefore \(\operatorname{red}_\Delta(s(w))=\mathfrak p_{\sigma_w}\).
\end{proof}

\begin{corollary}\label{cor:gauss-toric-strata}
The assignment \(F\mapsto\mathfrak p_{\sigma_F}\) is an
inclusion-preserving bijection from the nonempty faces of
\(\Delta^+\) to the torus-invariant primes of
\(k[\Sigma_{\Delta^+}]\). Moreover, if
\(w\in\operatorname{relint}(F)\cap\Delta^\uparrow\), then
\(\operatorname{red}_\Delta(s(w))=\mathfrak p_{\sigma_F}\).
\end{corollary}

\begin{proof}
The cones of \(\Sigma_{\Delta^+}\) are precisely the normal cones
\(\sigma_F\), and the torus-invariant primes of \(k[\Sigma_{\Delta^+}]\)
are precisely the corresponding \(\mathfrak p_{\sigma_F}\). Every nonempty
face \(F\) of \(\Delta^+\) meets \(\Delta\): if a linear functional exposes
\(F\), then it is nonnegative on \(\operatorname{rec}(\Delta^+)\), and a
minimizing point of \(\Delta\) also lies in \(F\). Since
\(\Delta\subseteq\Omega^\circ\) and \(\operatorname{relint}(F)\) is dense
in \(F\), it follows that
\(\operatorname{relint}(F)\cap\Delta^\uparrow\neq\varnothing\). The
reduction formula follows from
Theorem~\ref{thm:reduction-map}(ii). Finally, if \(F\subseteq F'\),
then \(\sigma_F\supseteq\sigma_{F'}\), and hence
\(\mathfrak p_{\sigma_F}\subseteq\mathfrak p_{\sigma_{F'}}\).
\end{proof}

\subsection{A two-vertex example}\label{subsec:two-vertex-example}

Consider \(D=D_1(K)\), with \(w_1=(-2,-1)\), \(w_2=(-1,-2)\), and
\(\Delta=\operatorname{conv}\{w_1,w_2\}\). For
\(m=(a,b)\in\mathbb R_{\geq0}^2\),
\[
\lambda_\Delta(a,b)=\begin{cases}-2a-b,&a\geq b,\\-a-2b,&b\geq a.
\end{cases}
\]
Thus \(\Sigma_{\Delta^+}\) has maximal cones
\(\sigma_1=\{a\geq b\geq0\}\) and
\(\sigma_2=\{b\geq a\geq0\}\), meeting along
\(\rho=\mathbb R_{\geq0}(1,1)\). Its other rays are
\(\rho_1=\mathbb R_{\geq0}(1,0)\) and
\(\rho_2=\mathbb R_{\geq0}(0,1)\).

Recall that \(e_{(a,b)}=\pi^{-\lambda_\Delta(a,b)}x^a\partial^b\),
and write \(z^{(a,b)}=\overline e_{(a,b)}\). In particular,
\(e_{(1,0)}=\pi^2x\), \(e_{(0,1)}=\pi^2\partial\), and
\(e_{(1,1)}=\pi^3x\partial\). Since \((1,0)\) and \((0,1)\) lie in
no common cone, \(z^{(1,0)}z^{(0,1)}=0\). The special fiber is generated
by \(z^{(1,0)},z^{(0,1)},z^{(1,1)}\) with this relation, and hence is
the union of two affine planes meeting along an affine line.

The face-normal-cone correspondence and
Corollary~\ref{cor:gauss-toric-strata} identify all torus-invariant primes.
In the table, \(w\) may be any point of
\(\operatorname{relint}(F)\cap\Delta^\uparrow\).

\begin{center}
\begingroup
\small
\renewcommand{\arraystretch}{1.4}
\setlength{\tabcolsep}{8pt}
\begin{tabular}{lll}
\hline
Face \(F\) of \(\Delta^+\) & \(\sigma_F\) &
\(\operatorname{red}_\Delta(s(w))=\mathfrak p_{\sigma_F}\)\\
\hline
\(\{w_1\}\) & \(\sigma_1\) & \((z^{(0,1)})\)\\
\(\{w_2\}\) & \(\sigma_2\) & \((z^{(1,0)})\)\\
\(\Delta\) & \(\rho\) & \((z^{(1,0)},z^{(0,1)})\)\\
\(w_1+\mathbb R_{\geq0}(0,1)\) & \(\rho_1\) & \((z^{(0,1)},z^{(1,1)})\)\\
\(w_2+\mathbb R_{\geq0}(1,0)\) & \(\rho_2\) & \((z^{(1,0)},z^{(1,1)})\)\\
\(\Delta^+\) & \(\{0\}\) &
\((z^{(1,0)},z^{(0,1)},z^{(1,1)})\)\\
\hline
\end{tabular}
\endgroup
\end{center}

The first two primes are the generic points of the irreducible components,
while \(\mathfrak p_\rho\) is the generic point of their intersection.
The remaining rows exhibit the other torus-invariant primes of the special
fiber.

\section{Polyhedral \texorpdfstring{\(G\)}{G}-site and structure sheaves}
\label{sec:polyhedral-site}

We now organize the polyhedral domains in the seminorm space
\(\mathcal S(D)\) into a \(G\)-site and attach to them the Banach and
Fr\'echet algebras constructed in Section~\ref{sec:polyhedral-completions}.
For a real polyhedron \(\Delta\subseteq\Omega\), set
\begin{equation}\label{eq:polyhedral-domain}
U_\Delta=\operatorname{trop}^{-1}(\Delta)\subseteq\mathcal S(D),
\end{equation}
and put \(U_\varnothing=\varnothing\).

\subsection{Polyhedral domains and admissible coverings}

Let \(\mathcal U_{\mathbb R}\) be the poset category of the domains
\(U_\Delta\subseteq\mathcal S(D)\), for real polyhedra
\(\Delta\subseteq\Omega\), ordered by inclusion. Since the Gauss
section is a section of tropicalization,
\(\operatorname{trop}(U_\Delta)=\Delta\). Thus
\(\Delta\mapsto U_\Delta\) identifies \(\mathcal U_{\mathbb R}\) with
the poset of real polyhedra in \(\Omega\). Fiber products are intersections:
\[
U_{\Delta_1}\times_{U_\Delta}U_{\Delta_2}
=U_{\Delta_1\cap\Delta_2}.
\]
Let \(\mathcal U\) be the full subcategory corresponding to
\(\Gamma\)-rational polyhedra. By
Section~\ref{subsec:polyhedral-background}, it is stable under these
fiber products.

A family \(\{U_{\Delta_i}\to U_\Delta\}_{i\in I}\) is an
\emph{admissible polyhedral covering} if the \(\Delta_i\) cover
\(\Delta\) and every compact polyhedron \(C\subseteq\Delta\) is
contained in the union of finitely many \(\Delta_i\). If \(\Delta\) is
compact, this is equivalent to the existence of a finite subcover. A nested
cofinal exhaustion of an arbitrary polyhedron by compact polyhedra is
therefore admissible; see Section~\ref{subsec:polyhedral-background}.

\begin{proposition}\label{prop:polyhedral-pretopology}
Admissible polyhedral coverings define a Grothendieck pretopology on
\(\mathcal U_{\mathbb R}\). The full subcategory \(\mathcal U\)
is an induced subsite in the sense of
Section~\ref{subsec:sites-sheaves-background}.
\end{proposition}

\begin{proof}
Identity coverings are admissible. Suppose that \(\Delta'\subseteq\Delta\)
and that \(\{\Delta_i\}_{i\in I}\) is an admissible covering of
\(\Delta\). Then the \(\Delta'\cap\Delta_i\) cover \(\Delta'\).
Every compact polyhedron \(C\subseteq\Delta'\) is covered by finitely
many \(\Delta_i\), and hence by the corresponding intersections. Thus
admissibility is stable under pullback.

For composition, suppose that \(\{\Delta_i\}_{i\in I}\) is an
admissible covering of \(\Delta\), and that
\(\{\Delta_{ij}\}_{j\in J_i}\) is an admissible covering of
\(\Delta_i\) for each \(i\). Given a compact polyhedron
\(C\subseteq\Delta\), choose finitely many \(\Delta_i\) covering
\(C\). For each such \(i\), the compact polyhedron
\(C\cap\Delta_i\) is covered by finitely many \(\Delta_{ij}\).
Hence finitely many members of the composite family cover \(C\).
\end{proof}

\subsection{The analytic structure sheaf}

For a nonempty real polyhedron \(\Delta\subseteq\Omega\), define
\begin{equation}\label{eq:structure-presheaf}
\mathcal O_{\mathbb R}(U_\Delta)=D_\Delta^{\mathrm{Fr}},
\qquad\mathcal O_{\mathbb R}(U_\varnothing)=0.
\end{equation}
If \(\Delta'\subseteq\Delta\), the natural continuous inclusion
\(D_\Delta^{\mathrm{Fr}}\to D_{\Delta'}^{\mathrm{Fr}}\) defines the
restriction homomorphism. Thus \(\mathcal O_{\mathbb R}\) is a presheaf
of \(K\)-algebras, with Fr\'echet section algebras and continuous
restriction homomorphisms. If \(\Delta\) is
compact, then \(\mathcal O_{\mathbb R}(U_\Delta)=D_\Delta\).

The sheaf property reduces first to finite coverings of compact polyhedra.
Empty members of a covering may be omitted.

\begin{proposition}\label{prop:compact-structure-gluing}
The presheaf \(\mathcal O_{\mathbb R}\) has the gluing property for
every finite covering of a compact polyhedral domain.
\end{proposition}

\begin{proof}
The empty case is immediate. Let
\(\Delta=\Delta_1\cup\cdots\cup\Delta_r\) be a covering with
\(\Delta\neq\varnothing\); the nonempty \(\Delta_i\) are compact because they are closed subsets of
\(\Delta\). Let \(f_i\in D_{\Delta_i}\) be compatible.
Realize all \(D_{\Delta_i}\) and \(D_{\Delta_i\cap\Delta_j}\), for
nonempty intersections, in the common space of formal ordered-monomial
series, as in Section~\ref{sec:polyhedral-completions}. The restriction
maps send a series to the same series. Hence, if
\(\Delta_i\cap\Delta_j\neq\varnothing\), compatibility means that
the images of \(f_i\) and \(f_j\) in
\(D_{\Delta_i\cap\Delta_j}\) are equal, and therefore their
ordered-monomial coefficients agree.

The graph with vertices \(1,\ldots,r\), and an edge whenever
\(\Delta_i\cap\Delta_j\neq\varnothing\), is connected. Otherwise
\(\Delta\) would be the disjoint union of two nonempty closed subsets,
contradicting its convexity. Hence all \(f_i\) have the same coefficients.
Write the resulting formal series as
\(f=\sum_{\alpha,\beta}c_{\alpha,\beta}x^\alpha\partial^\beta\).
For every ordered monomial, we have
\[
\bigl|c_{\alpha,\beta}x^\alpha\partial^\beta\bigr|_\Delta
=\max_i\bigl|c_{\alpha,\beta}x^\alpha\partial^\beta\bigr|_{\Delta_i}.
\]
Since the covering is finite and the monomial terms tend to zero in each
\(D_{\Delta_i}\), they tend to zero for the \(\Delta\)-norm. Thus
\(f\in D_\Delta\). Uniqueness follows from the ordered-monomial expansion.
\end{proof}

\begin{theorem}\label{thm:structure-sheaf}
The presheaf \(\mathcal O_{\mathbb R}\) is a sheaf of \(K\)-algebras
on \(\mathcal U_{\mathbb R}\). Its section algebras are complete
nonarchimedean Fr\'echet \(K\)-algebras, and its restriction homomorphisms
are continuous.
\end{theorem}

\begin{proof}
Let \(\{U_{\Delta_i}\to U_\Delta\}_{i\in I}\) be an admissible
polyhedral covering, and let \(f_i\in D_{\Delta_i}^{\mathrm{Fr}}\) be
a compatible family. The empty object is immediate, so assume that
\(\Delta\) is nonempty.

For every nonempty compact polyhedron \(C\subseteq\Delta\), admissibility
provides finitely many \(\Delta_i\) such that the nonempty compact
polyhedra \(C\cap\Delta_i\) cover \(C\).
Proposition~\ref{prop:compact-structure-gluing} glues the corresponding
restrictions of the \(f_i\) to a unique element \(f_C\in D_C\).
This element is independent of the chosen finite subcover, since two choices
have a common finite refinement and compact gluing is unique.

If \(C'\subseteq C\), the same uniqueness shows that the restriction of
\(f_C\) to \(D_{C'}\) is \(f_{C'}\). Hence the \(f_C\) define an
element
\[
f\in\varprojlim_{C\subseteq\Delta}D_C=D_\Delta^{\mathrm{Fr}},
\]
where \(C\) ranges over the nonempty compact polyhedra contained in
\(\Delta\). Its restriction to every \(U_{\Delta_i}\) is \(f_i\),
since this may be checked after restriction to compact subpolyhedra of
\(\Delta_i\). Uniqueness follows in the same way.
\end{proof}

\begin{samepage}
\begin{remark}\label{rem:evaluation-norming}
Let \(\Delta\subseteq\Omega\) be a nonempty real polyhedron.
\begin{enumerate}[label=\textup{(\roman*)}]
\item If \(\eta\in U_\Delta\subseteq\mathcal S(D)\) and
\(w=\operatorname{trop}(\eta)\), then \(\eta(P)\leq|P|_w\) for
\(P\in D\). Hence \(\eta\) extends continuously to \(D_w\), and the
restriction \(D_\Delta^{\mathrm{Fr}}\to D_w\) defines \(\eta(f)\)
for every \(f\in\mathcal O_{\mathbb R}(U_\Delta)\).
\item If \(C\subseteq\Delta\) is nonempty and compact, then for every
\(f\in D_C\),
\begin{equation}\label{eq:section-norming}
|f|_C=\sup_{\eta\in U_C}\eta(f)=\sup_{w\in C}s(w)(f).
\end{equation}
Thus \(s(C)\) is norming for \(D_C\) in the sense of
Section~\ref{subsec:toric-background}.
\end{enumerate}
\end{remark}
\end{samepage}

\begin{proof}
Part (i) follows from the maximality of the Gauss point in each tropical
fiber. For (ii), the vertex formula shows that
\(|f|_C=\sup_{w\in C}s(w)(f)\). Part (i) implies
\(\eta(f)\leq s(\operatorname{trop}(\eta))(f)\leq|f|_C\) for
\(\eta\in U_C\), while \(s(C)\subseteq U_C\).
\end{proof}

\subsection{The integral structure sheaf}\label{subsec:integral-structure-sheaf}

Recall from Section~\ref{subsec:toric-background} that
\(K^\circ=\{a\in K:|a|\leq1\}\) is the valuation ring of \(K\).
For a nonempty real polyhedron \(\Delta\subseteq\Omega\), define
\begin{equation}\label{eq:integral-structure-sheaf}
\mathcal O_{\mathbb R}^{\circ}(U_\Delta)
=\{f\in\mathcal O_{\mathbb R}(U_\Delta):
\eta(f)\leq1\text{ for every }\eta\in U_\Delta\},
\end{equation}
and put \(\mathcal O_{\mathbb R}^{\circ}(U_\varnothing)=0\).

\begin{proposition}\label{prop:integral-structure-sheaf}
The assignment \(\mathcal O_{\mathbb R}^{\circ}\) is a subsheaf of
\(K^\circ\)-algebras of \(\mathcal O_{\mathbb R}\). If \(\Delta\) is
nonempty and compact, then
\begin{equation}\label{eq:compact-integral-sections}
\mathcal O_{\mathbb R}^{\circ}(U_\Delta)=D_\Delta^\circ.
\end{equation}
\end{proposition}

\begin{proof}
Restrictions preserve integrality. Suppose that integral local sections
on an admissible polyhedral covering of \(U_\Delta\) glue in
\(\mathcal O_{\mathbb R}\) to a section \(f\). For
\(\eta\in U_\Delta\), choose a member of the covering containing
\(\eta\). The restriction of \(f\) there is integral, so \(\eta(f)\leq1\).
Hence \(f\in\mathcal O_{\mathbb R}^{\circ}(U_\Delta)\).

If \(\Delta\) is compact, Remark~\ref{rem:evaluation-norming}(ii) shows
that the condition \(\eta(f)\leq1\) for every \(\eta\in U_\Delta\)
is equivalent to \(|f|_\Delta\leq1\). This proves
\eqref{eq:compact-integral-sections}.
\end{proof}

Restricting to the induced \(\Gamma\)-rational subsite \(\mathcal U\), set
\[
\mathcal O=\mathcal O_{\mathbb R}|_{\mathcal U},
\qquad
\mathcal O^\circ=\mathcal O_{\mathbb R}^\circ|_{\mathcal U}.
\]
By Proposition~\ref{prop:polyhedral-pretopology} and
Section~\ref{subsec:sites-sheaves-background}, these are sheaves on
\(\mathcal U\). This restriction does not require a discrete valuation.

\begin{remark}\label{rem:integral-model-comparison}
Suppose that \(K\) is discretely valued, and write \(R=K^\circ\), with
notation as in Section~\ref{sec:flatness}. If \(\Delta\) is nonempty,
compact, and \(\Gamma\)-rational,
Proposition~\ref{prop:integral-completion} identifies
\[
\mathcal O^\circ(U_\Delta)=D_\Delta^\circ\simeq\widehat A_\Delta.
\]
If, in addition, \(\Delta\) has integral vertices and
\(\Delta\subseteq\Omega^\circ\), Theorem~\ref{thm:toric-face-reduction}
identifies
\[
\mathcal O^\circ(U_\Delta)/\pi\mathcal O^\circ(U_\Delta)
\simeq k[\Sigma_{\Delta^+}],
\]
where \(\Sigma_{\Delta^+}\) is the minimum normal fan of the upper hull
\(\Delta^+=\Delta+\mathbb R_{\geq0}^{2n}\).
\end{remark}

\section{Module sheaves and tropical \texorpdfstring{\(D\)}{D}-ideals}
\label{sec:module-sheaves}

Assume throughout this section that \(K\) is discretely valued, with
notation as in Section~\ref{sec:flatness}. We work on the induced
\(\Gamma\)-rational subsite \(\mathcal U\) and with the structure sheaf
\(\mathcal O\) of Section~\ref{sec:polyhedral-site}.

Let \(M\) be a finitely generated left \(D\)-module. For a nonempty
\(\Gamma\)-rational polyhedron \(\Delta\subseteq\Omega\), define
\begin{equation}\label{eq:module-presheaf}
\widetilde M(U_\Delta)=D_\Delta^{\mathrm{Fr}}\otimes_DM,
\end{equation}
and put \(\widetilde M(U_\varnothing)=0\). If
\(\Delta'\subseteq\Delta\), the homomorphism
\(D_\Delta^{\mathrm{Fr}}\to D_{\Delta'}^{\mathrm{Fr}}\) induces the
restriction map on \(\widetilde M\). Thus \(\widetilde M\) is a presheaf
of left \(\mathcal O\)-modules.

If \(\Delta\) is compact, then
\(\widetilde M(U_\Delta)=D_\Delta\otimes_DM\). More generally,
Proposition~\ref{prop:finite-module-limit} identifies, for every nested
cofinal compact \(\Gamma\)-rational exhaustion \((C_j)\) of \(\Delta\),
\begin{equation}\label{eq:module-presheaf-limit}
\widetilde M(U_\Delta)\simeq\varprojlim_j(D_{C_j}\otimes_DM)
\end{equation}
as topological modules. If \(M\simeq D^r/N\),
\eqref{eq:frechet-presentation} identifies the same module with the
Hausdorff quotient \((D_\Delta^{\mathrm{Fr}})^r/D_\Delta^{\mathrm{Fr}}N\).

\begin{lemma}\label{lem:compact-mayer-vietoris}
Let \(\Delta=\Delta_1\cup\Delta_2\) be a compact
\(\Gamma\)-rational polyhedron, where \(\Delta_1,\Delta_2\) are nonempty
compact \(\Gamma\)-rational polyhedra, and put
\(C=\Delta_1\cap\Delta_2\). Then \(C\neq\varnothing\), and
\begin{equation}\label{eq:compact-mayer-vietoris}
0\longrightarrow D_\Delta\longrightarrow D_{\Delta_1}\oplus D_{\Delta_2}
\longrightarrow D_C\longrightarrow0
\end{equation}
is exact, where the last map is \((f_1,f_2)\mapsto f_1-f_2\).
Consequently,
\begin{equation}\label{eq:module-mayer-vietoris}
0\longrightarrow D_\Delta\otimes_DM
\longrightarrow(D_{\Delta_1}\otimes_DM)\oplus(D_{\Delta_2}\otimes_DM)
\longrightarrow D_C\otimes_DM\longrightarrow0
\end{equation}
is exact.
\end{lemma}

\begin{proof}
Convexity of \(\Delta\) and closedness of the two nonempty covering
polyhedra imply that \(C\neq\varnothing\).
Proposition~\ref{prop:compact-structure-gluing} proves exactness of
\eqref{eq:compact-mayer-vietoris} at the first two terms. For
\(m\in\mathbb N^{2n}\), we claim that
\begin{equation}\label{eq:intersection-support-function}
\lambda_C(m)=\max\{\lambda_{\Delta_1}(m),\lambda_{\Delta_2}(m)\}.
\end{equation}
The inequality \(\geq\) follows from \(C\subseteq\Delta_1,\Delta_2\).
For the reverse inequality, suppose
\(a_1:=\lambda_{\Delta_1}(m)\leq a_2:=\lambda_{\Delta_2}(m)\), and
choose \(w_i\in\Delta_i\) with \(\langle w_i,m\rangle=a_i\).
Let \(w(t)=(1-t)w_1+tw_2\), and let \(t_0\) be the smallest
\(t\in[0,1]\) for which \(w(t)\in\Delta_2\). If \(t_0=0\), then
\(w(t_0)=w_1\in C\). Otherwise, since
\(\Delta=\Delta_1\cup\Delta_2\) is convex, \(w(t)\in\Delta_1\) for
\(t<t_0\), and closedness implies \(w(t_0)\in C\). Hence
\(a_2\leq\langle w(t_0),m\rangle=(1-t_0)a_1+t_0a_2\leq a_2\),
proving \eqref{eq:intersection-support-function}.

Let \(f=\sum_m c_mx^\alpha\partial^\beta\in D_C\), where
\(m=(\alpha,\beta)\). Write \(f=f_1+f_2\), where \(f_1\) is the
sum of the terms for which
\(\lambda_{\Delta_1}(m)\geq\lambda_{\Delta_2}(m)\), and \(f_2\)
is the sum of the remaining terms. By
\eqref{eq:intersection-support-function}, \(\lambda_C=\lambda_{\Delta_1}\)
on the first set of exponents and \(\lambda_C=\lambda_{\Delta_2}\) on
the second. The defining convergence condition for \(D_C\) therefore
implies \(f_1\in D_{\Delta_1}\) and \(f_2\in D_{\Delta_2}\).
The pair \((f_1,-f_2)\) maps to \(f\), proving surjectivity of the
last map in \eqref{eq:compact-mayer-vietoris}.

Applying \(-\otimes_DM\) to \eqref{eq:compact-mayer-vietoris}, right
exactness preserves exactness from the middle term onward. Since \(D_C\)
is flat as a right \(D\)-module by Theorem~\ref{thm:compact-flat},
\(\operatorname{Tor}_1^D(D_C,M)=0\), so the first map remains injective.
This proves \eqref{eq:module-mayer-vietoris}.
\end{proof}

In particular, \(\widetilde M\) has the gluing property of
Section~\ref{subsec:sites-sheaves-background} for binary compact
\(\Gamma\)-rational coverings.

\begin{proposition}\label{prop:compact-module-gluing}
The presheaf \(\widetilde M\) has the gluing property for every finite
covering of a compact \(\Gamma\)-rational polyhedral domain.
\end{proposition}

\begin{proof}
Let \(\Delta=\Delta_1\cup\cdots\cup\Delta_r\) be such a covering.
The empty case is immediate, and empty members may be omitted.
Lemma~\ref{lem:compact-mayer-vietoris} establishes the gluing property for
a covering by two compact polyhedra. A direct induction on the number of
members of the present covering is not available, since a partial union
such as \(\Delta_1\cup\cdots\cup\Delta_{r-1}\) need not be a polyhedron.
We therefore refine the covering by a common polyhedral subdivision obtained
from successive hyperplane cuts, so that
Lemma~\ref{lem:compact-mayer-vietoris} can be applied at each step.

For each \(\Delta_i\), choose a finite system of \(\Gamma\)-rational
affine-linear inequalities defining it, and let \(H_1,\ldots,H_N\) be
the affine hyperplanes underlying these inequalities. Write \(H_j^+\) and
\(H_j^-\) for the two closed half-spaces bounded by \(H_j\).

Set \(\mathcal T_0=\{\Delta\}\). Inductively, obtain \(\mathcal T_j\)
from \(\mathcal T_{j-1}\) by replacing each \(C\in\mathcal T_{j-1}\)
with the nonempty members among \(C\cap H_j^+\) and \(C\cap H_j^-\).
Thus \(\mathcal T_j\) is a finite subdivision of \(\Delta\) by compact
\(\Gamma\)-rational polyhedra. Every cell \(C\in\mathcal T_N\) is
contained in some \(\Delta_i\). Indeed, choose a point in the relative
interior of \(C\). It lies in some \(\Delta_i\), and the signs of the
affine-linear functions defining \(\Delta_i\) are fixed on the relative
interior of \(C\); since \(\Delta_i\) is closed, \(C\subseteq\Delta_i\).

Let \(m_i\in\widetilde M(U_{\Delta_i})\) be compatible. For each
\(C\in\mathcal T_N\), choose \(i\) with \(C\subseteq\Delta_i\) and
set \(m_C=m_i|_C\). Compatibility of the \(m_i\) shows that \(m_C\)
is independent of this choice and that the resulting sections agree on
intersections.

We now pass successively from \(\mathcal T_j\) to \(\mathcal T_{j-1}\),
for \(j=N,\ldots,1\). Let \(C\in\mathcal T_{j-1}\). If both
\(C^+=C\cap H_j^+\) and \(C^-=C\cap H_j^-\) are nonempty, then
\(C=C^+\cup C^-\) and \(C^+\cap C^-=C\cap H_j\). The sections
already constructed on \(C^+\) and \(C^-\) agree on their intersection,
so Lemma~\ref{lem:compact-mayer-vietoris} determines a unique section
\(m_C\in\widetilde M(U_C)\) restricting to both. If one of \(C^\pm\)
is empty, the other equals \(C\), so no gluing is required.

The sections obtained on the cells of \(\mathcal T_{j-1}\) remain
compatible. Indeed, if \(C,C'\in\mathcal T_{j-1}\), their restrictions
to \(C\cap C'\) agree after intersecting with \(H_j^+\) and with
\(H_j^-\), by compatibility on \(\mathcal T_j\). The uniqueness part
of Lemma~\ref{lem:compact-mayer-vietoris}, applied to the corresponding
binary covering of \(C\cap C'\), shows that the restrictions agree on all
of \(C\cap C'\); if either intersection is empty, equality follows from
the other piece.

Repeating this procedure produces a section
\(m\in\widetilde M(U_\Delta)\). By construction, \(m\) and \(m_i\)
agree on the subdivision of \(\Delta_i\) induced by \(\mathcal T_N\);
reversing the same sequence of cuts inside \(\Delta_i\) and using
uniqueness in Lemma~\ref{lem:compact-mayer-vietoris} shows that
\(m|_{\Delta_i}=m_i\). The same argument proves uniqueness.
\end{proof}

\begin{corollary}\label{thm:module-sheaf}
For every finitely generated left \(D\)-module \(M\), the presheaf
\(\widetilde M\) is a sheaf of left \(\mathcal O\)-modules on
\(\mathcal U\).
\end{corollary}

\begin{proof}
Let \(\{U_{\Delta_i}\to U_\Delta\}_{i\in I}\) be an admissible
polyhedral covering, and let \(m_i\in\widetilde M(U_{\Delta_i})\)
be compatible. For nonempty \(\Delta\), choose, as in
Section~\ref{subsec:polyhedral-background}, a nested cofinal compact
\(\Gamma\)-rational exhaustion
\(C_1\subseteq C_2\subseteq\cdots\) of \(\Delta\).

For each \(j\), admissibility provides finitely many \(\Delta_i\)
covering \(C_j\). The nonempty intersections \(C_j\cap\Delta_i\)
form a finite covering of \(C_j\) by compact \(\Gamma\)-rational
polyhedra. Proposition~\ref{prop:compact-module-gluing} therefore
determines a unique section \(m_j\in D_{C_j}\otimes_DM\) whose
restrictions agree with the corresponding \(m_i\). This section is
independent of the finite subcover: two choices have a common finite
refinement, and uniqueness applies there. The same argument shows that
\(m_{j+1}|_{C_j}=m_j\). Proposition~\ref{prop:finite-module-limit}
identifies the compatible family \((m_j)\) with a unique section
\[
m\in D_\Delta^{\mathrm{Fr}}\otimes_DM=\widetilde M(U_\Delta).
\]
For every \(i\), restriction to compact \(\Gamma\)-rational
subpolyhedra of \(\Delta_i\) shows that \(m|_{\Delta_i}=m_i\).
Uniqueness follows in the same way. The empty object has zero sections.
\end{proof}

\begin{theorem}\label{thm:module-sheaf-exact}
The assignment \(M\mapsto\widetilde M\) is an exact functor from the
category of finitely generated left \(D\)-modules to the category of
sheaves of left \(\mathcal O\)-modules on \(\mathcal U\).
\end{theorem}

\begin{proof}
Let \(0\to M'\to M\to M''\to0\) be an exact sequence of finitely
generated left \(D\)-modules. For every nonempty \(\Gamma\)-rational
polyhedron \(\Delta\subseteq\Omega\),
Corollary~\ref{cor:exact-base-change} shows that
\[
0\longrightarrow D_\Delta^{\mathrm{Fr}}\otimes_DM'
\longrightarrow D_\Delta^{\mathrm{Fr}}\otimes_DM
\longrightarrow D_\Delta^{\mathrm{Fr}}\otimes_DM''\longrightarrow0
\]
is exact. These maps commute with restriction in \(\Delta\), so
\[
0\longrightarrow\widetilde{M'}\longrightarrow\widetilde M
\longrightarrow\widetilde{M''}\longrightarrow0
\]
is an exact sequence of sheaves.
\end{proof}

We conclude by specializing this construction to cyclic modules associated
with tropical \(D\)-ideals. Let \(I\subseteq D\) be a left ideal.
By Theorem~\ref{thm:trop-ideal-polyhedral},
\(\operatorname{Trop}_D(I)\) is the support of a finite
\(\Gamma\)-rational polyhedral complex. Choose such a complex
\(\Sigma_I\), so that
\(|\Sigma_I|=\operatorname{Trop}_D(I)\), and set
\[
U_I=\operatorname{trop}^{-1}(|\Sigma_I|)\subseteq\mathcal S(D).
\]
In general \(U_I\) need not be a single object of \(\mathcal U\), since
\(|\Sigma_I|\) need not be convex. Let \(\mathcal U_I\) be the full
subcategory of \(\mathcal U\) whose objects are the \(U_\Delta\) with
\(\Delta\subseteq|\Sigma_I|\). Intersections of such polyhedra remain
in \(|\Sigma_I|\), so this full subcategory is stable under fiber
products. We equip it with the induced pretopology of
Section~\ref{subsec:sites-sheaves-background}. The domains corresponding to the cells of \(\Sigma_I\) form a finite
polyhedral covering of \(U_I\).

Restricting the structure sheaf, set
\(\mathcal O_I=\mathcal O|_{\mathcal U_I}\). By
Section~\ref{subsec:sites-sheaves-background}, this is a sheaf. Thus
\((\mathcal U_I,\mathcal O_I)\) is a ringed subsite of
\((\mathcal U,\mathcal O)\).

The cyclic module \(D/I\) determines the sheaf \(\widetilde{D/I}\).
Its restriction to \(\mathcal U_I\) is a sheaf of left
\(\mathcal O_I\)-modules, and for every nonempty
\(U_\Delta\in\mathcal U_I\),
\begin{equation}\label{eq:cyclic-module-sections}
\widetilde{D/I}(U_\Delta)\simeq
D_\Delta^{\mathrm{Fr}}/D_\Delta^{\mathrm{Fr}}I.
\end{equation}
Thus the tropical set selects the polyhedral subsite, while the
ideal-dependent analytic data are carried by the cyclic module sheaf.
Unlike a commutative defining ideal, a left ideal need not be two-sided
and therefore does not in general define a quotient structure sheaf.

\begingroup
\small
\hbadness=2000
\let\oldthebibliography\thebibliography
\renewcommand{\thebibliography}[1]{%
  \oldthebibliography{#1}%
  \setlength{\itemsep}{1.5pt}%
}
\bibliographystyle{plain}
\bibliography{references}
\endgroup

\end{document}